\documentclass[12pt,reqno]{amsart}

\usepackage[text={420pt,660pt},centering]{geometry}

\usepackage{amssymb,amsfonts,amsthm,mathrsfs}
\usepackage{marginnote}
\usepackage{enumitem}
\usepackage{graphicx}
\usepackage{bm}

\usepackage[dvipsnames]{xcolor}

\usepackage[colorlinks=true, pdfstartview=FitV, linkcolor=black, citecolor=blue, urlcolor=blue]{hyperref}

\theoremstyle{plain}
\newtheorem{theorem}{Theorem}[section]

\newtheorem{lemma}[theorem]{Lemma}

\newtheorem{proposition}[theorem]{Proposition}

\newtheorem{result}[theorem]{Observations}

\theoremstyle{definition}
\newtheorem{remark}[theorem]{Remark}

\newcommand{\R}{\mathbb{R}}

\newcommand{\N}{\mathbb{N}}

\newcommand{\supp}{\mathop{\mathrm{supp}}}

\newcommand{\M}{\mathcal{M}}
\newcommand{\A}{\mathcal{A}}

\newcommand{\p}{\partial}

\newcommand{\les}{\lesssim}

\let\div\relax
\DeclareMathOperator{\div}{div}

\let\tilde\relax
\newcommand{\tilde}[1]{\widetilde{#1}}
\let\hat\relax
\newcommand{\hat}[1]{\widehat{#1}}

\newcommand{\BMO}{{\rm BMO}}

\numberwithin{equation}{section}
\numberwithin{figure}{section}
\setlist[enumerate]{leftmargin=*}

\title{Forward self-similar solutions to the 2D Navier--Stokes equations}

\author[Albritton]{Dallas Albritton} 
\address[Dallas Albritton]{University of Wisconsin-Madison, Department of Mathematics,  480 Lincoln Dr, Madison, WI 53706, USA}
\email{dalbritton@wisc.edu}

\author[Guillod]{Julien Guillod}
\address[Julien Guillod]{Sorbonne Université, CNRS, Université Paris Cité, Inria, Laboratoire Jacques-Louis Lions (LJLL), 75005 Paris, France}
\address{École Normale Supérieure, Université PSL, Département de Mathématiques et applications, 75005, Paris, France}
\email{julien.guillod@sorbonne-universite.fr}

\author[Korobkov]{Mikhail Korobkov}
\address[Mikhail Korobkov]{Fudan University, School of Mathematical Sciences, Handan Road 220, 200433 Shanghai, China}
\email{korob@math.nsc.ru}

\author[Ren]{Xiao Ren}
\address[Xiao Ren]{Fudan University, Center for Applied Mathematics, Handan Road 220, 200433 Shanghai, China}
\email{xren@fudan.edu.cn}

\date\today

\begin{document}
\begin{abstract}
We construct self-similar solutions to the 2D Navier--Stokes equations evolving from arbitrarily large $-1$--homogeneous initial data and present numerical evidence for their non-uniqueness.
\end{abstract}

\maketitle

\setcounter{tocdepth}{1}
\tableofcontents

\parskip   2pt plus 0.5pt minus 0.5pt

\section{Introduction}

We consider the incompressible Navier--Stokes equations
\begin{equation}
\label{eq:NS}
\tag{NS}
\left\{
    \begin{aligned}
        &\p_t u + u \cdot \nabla u + \nabla p = \Delta u \\
        &\div u = 0
    \end{aligned}\right. 
\end{equation}
in the whole space $\R^d$, $d=2,3$. Recently, there has been significant progress in our understanding of uniqueness for solutions to the Cauchy problem for~\eqref{eq:NS} in 3D: We now know that Leray's weak solutions~\cite{leray1934} are not unique. 
This was anticipated by Ladyzhenskaya in~\cite{ladyzhenskayanonuniqueness}. Two concrete non-uniqueness scenarios were proposed by Jia and \v{S}ver{\'a}k in~\cite{jiasverakillposed}. Numerical evidence was subsequently presented by the second author and \v{S}ver{\'a}k in~\cite{guillodsverak}. A rigorous proof was obtained with non-zero external force by the first author, Bru{\`e}, and Colombo in~\cite{albritton2021non}. Finally, a computer-assisted proof was announced by T.~Hou, Y.~Wang, and C.~Yang in~\cite{hou2025nonuniqueness}.

Key to this progress has been the investigation of \emph{self-similar solutions}. The Navier--Stokes equations are invariant under the scaling transformation
\begin{equation}\label{eq:scaling}
    u_\ell(x,t) = \frac{1}{\ell} u \left( \frac{x}{\ell} , \frac{t}{\ell^2} \right) \, , \quad p_\ell(x,t) = \frac{1}{\ell^2} p\left( \frac{x}{\ell} , \frac{t}{\ell^2} \right) \, ,  \quad \ell > 0 \, ,
\end{equation}
corresponding to the dimension counting~\cite[(1.9)]{ckn}
\begin{equation}
    [x] = L \, , \quad [t] = L^2 \, , \quad [u] = L^{-1} \, ,  \quad [p] = L^{-2} \, ,
\end{equation}
obtained by identifying length${}^2$ and time using the viscosity $\nu$, which has dimensions $[\nu] = L^2/T$. A solution $(u,p)$ to~\eqref{eq:NS} in $\R_+ \times \R^d$, $d=2,3$, is a (forward) \emph{self-similar solution} if it is invariant under the scaling symmetry~\eqref{eq:scaling}. Such solutions are necessarily of the form
\begin{equation} \label{eq:self-similar}
    u(x,t) = \frac{1}{\sqrt{t}} U \left( \frac{x}{\sqrt{t}} \right) \, , \quad p(x,t) = \frac{1}{t} P \left( \frac{x}{\sqrt{t}} \right) \, .
\end{equation}
In particular, $(u,p)$ is entirely determined by its profile $(U,P)$ at time $t=1$. Plugging~\eqref{eq:self-similar} into the Navier--Stokes equations, we obtain the \emph{Leray equations}
\begin{equation}
    \label{eq:LerayEquations}
    \left\{
\begin{aligned}
	&- \Delta U - \frac{1}{2} (U + y \cdot \nabla_y U) + U \cdot \nabla U + \nabla P = 0 \\
    &\div U = 0
\end{aligned}\right.
\end{equation}
where $y = x/\sqrt{t}$. 
Self-similar solutions evolve from $-1$--homogeneous initial data $u_0$; the initial condition $u \to u_0$ as $t \to 0^+$ is encoded in the boundary condition
\begin{equation}
    \label{eq:LerayEquationsBC}
    U(y) = u_0(y) + o \left( \frac{1}{|y|} \right) \text{ as } |y| \to +\infty \, .
\end{equation}

Non-trivial $-1$--homogeneous velocity fields do not belong to the critical space $L^d$, where local-in-time well-posedness holds~\cite{kato}; rather, they belong to the weak Lebesgue space $L^{d,\infty}$, where only small-data global well-posedness results are known.\footnote{Informally, critical spaces are those for which the norm is invariant under the scaling in~\eqref{eq:scaling}. The space $L^{d,\infty}$ has the additional feature that \emph{$C^\infty_0$ functions are not dense}, which translates to small-data global well-posedness, not local well-posedness for arbitrary data, under the arguments in, e.g.,~\cite{kato,KochTataruWellposednessNavierStokes} or ~\cite[Chapter 5]{TsaiBook}. Such spaces were loosely termed \emph{ultra-critical} in~\cite{BedrossianGermainHarropGriffithsFilaments}.} Therefore, we regard $-1$--homogeneous velocity fields as on the borderline of the well-posedness theory. We examine this borderline through the lens of self-similar solutions. At a glance, this approach seems more tractable, since~\eqref{eq:LerayEquations} is elliptic and bears some resemblance to the steady Navier--Stokes equations, for which non-uniqueness of solutions to the boundary-value problem is common, even expected, see, e.g., \cite{Yudovich1966,GaldiBook,KorobBook}.

Perhaps the simplest non-uniqueness scenario is that multiple self-similar solutions evolve from the same self-similar initial data $u_0$. In~\cite{jiasverakillposed}, Jia and \v{S}ver{\'a}k speculated that this might arise due to a bifurcation in the set of solutions to~\eqref{eq:LerayEquations}-\eqref{eq:LerayEquationsBC} as the size of $u_0$ is increased. Subsequently, such a bifurcation was reported numerically by the second author and \v{S}ver{\'a}k in~\cite{guillodsverak}.

The uniqueness question has been largely ignored in dimension $d=2$, since Leray-Hopf solutions are unique in 2D. However, in the infinite-energy class, there is nothing to prevent the above scenario from occurring, as we illustrate below.

\subsection{Main results}

In this paper, we consider self-similar solutions to~\eqref{eq:NS} in dimension $d=2$. We begin by investigating the following question:
\begin{quote}
\textbf{(Q1)} Given a $-1$--homogeneous initial velocity field, does there exist a self-similar solution to the Navier--Stokes equations?
\end{quote}

Our main theorem answers this question in the affirmative:

\begin{theorem} \label{thm:main}
Let $\alpha \in (0,1)$ and  $u_0 \in C^\alpha(\R^2 \setminus \{0\})$ be a $-1$--homogeneous divergence-free vector field. Then there exists a self-similar solution $u$ to the 2D Navier--Stokes equations with initial data $u_0$. Moreover, its profile $U$ is smooth and satisfies 
	\begin{equation}
		|U(y) - (e^\Delta u_0) (y)| \le C(\alpha, \M) \langle y \rangle^{-1-\alpha} \, ,
	\end{equation}
where $\M := \|u_0\|_{C^\alpha(S_1)}$. 
\end{theorem}

The analogous theorem in 3D was obtained by Jia and \v{S}ver{\'a}k in~\cite{jiasverakselfsim}. However, the 2D problem carries significant new difficulties, foremost among which is that the initial data $u_0 \sim \frac{1}{|x|}$ has infinite energy, even locally; hence, the problem lacks any clear \emph{a priori} estimate, unlike the 3D case. Indeed, we derive \emph{a priori} estimates that hold specifically for self-similar solutions and not for general solutions to the evolutionary problem.

Once the existence of self-similar solutions is established, the natural question is the following:
\begin{quote}
\textbf{(Q2)} Are the self-similar solutions to the Navier--Stokes equations unique?
\end{quote}
Our numerical results indicate that the answer is negative:
\begin{result}\label{numerics}
	For the $-1$--homogeneous initial velocity field,
	\begin{equation}
		u_{0}(y) = -\frac{\sigma y_{1}y}{|y|^{3}},
	\end{equation}
	we numerically observe the existence of three different self-similar solutions to the 2D Navier--Stokes equations with initial data $u_0$ for $39.2\lesssim\sigma\leq80$.
\end{result}
The analogous result was obtained in 3D by the second author and \v{S}ver{\'a}k in~\cite{guillodsverak}. The numerical strategy in 2D is very similar to the 3D one. While continuing the solution in $\sigma$, we compute the eigenvalues of the linearization around the solution, and by standard bifurcation argument, if an eigenvalue becomes zero, then a new solution emerges. In both 2D and 3D, the scenario is a pitchfork bifurcation through the breaking of a discrete symmetry, as shown on Figure~\ref{fig:Uintro}. We refer the reader to Section~\ref{sec:numericalobservations} for further discussion.

\begin{figure}[h]
	\includegraphics{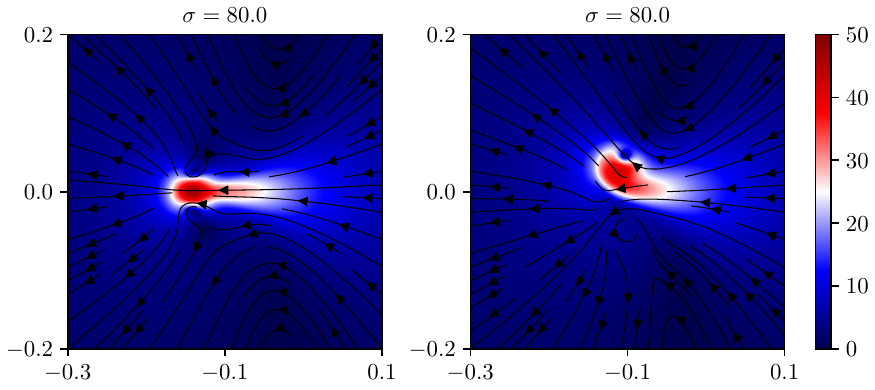}
	\caption{Symmetric and non-symmetric solutions with $\sigma = 80$.}
    \label{fig:Uintro}
\end{figure}

Finally, while computations have played an important role in the non-uniqueness theory, both at the level of numerical evidence and computer-assisted proof~\cite{hou2025nonuniqueness}, it remains interesting to develop such a theory by hand. The investigation of large self-similar solutions in Theorem~\ref{thm:main} may be viewed as a step in this direction.

\subsection{Strategy} To construct arbitrarily large solutions $U$, the two main ingredients will be \emph{a priori} estimates and compactness, the key assumptions to apply a version of the Leray-Schauder fixed point theory.

We begin by highlighting notable differences between the 2D and 3D theory, starting with the available \emph{a priori} estimates:
\begin{enumerate}[label=(\roman*)]
\item The kinetic energy $\frac{1}{2} \int |u|^2 \, dx$ is scaling-critical in 2D but supercritical in 3D.
\item In 2D, the vorticity $\omega := \nabla^\perp \cdot u$ satisfies the advection-diffusion equation
\begin{equation}
    \label{eq:vorticityequation}
    \p_t \omega + u \cdot \nabla \omega = \Delta \omega \, .
\end{equation}
Hence, every $L^p$ norm $\| \omega \|_{L^p}$, $1 \leq p \leq +\infty$, is monotonically decreasing under the evolution in the absence of boundary.\footnote{More generally, $\int f(\omega) \, dx$ will be decreasing for convex functions $f$.}
\end{enumerate}
The 2D Navier--Stokes equations are commonly viewed as better behaved, since the \emph{a priori} estimates for the time-evolution problem are better from a regularity perspective. These expectations are subverted for $-1$--homogeneous initial data:\footnote{The similar ``subversion" holds, e.g., for the stationary flow-around-obstacle problem, which was basically solved by J.~Leray himself for 3D case, but is still open for 2D, see, e.g., \cite{korob-ren23}.}

First, non-zero $-1$--homogeneous initial data $u_0$ always has infinite energy. A key point in 3D constructions, beginning with~\cite{jiasverakselfsim}, is that $u_0$ has \emph{locally finite energy}, which furnishes \emph{a priori} estimates for the evolution problem~\cite{LemarieRieussetRecentDevelopments}. However, in 2D, $u_0$ has locally infinite energy.

Second, for non-zero $-1$--homogeneous initial data, the vorticity $\omega_0$, as a distribution, is $-2$--homogeneous and therefore never represented by a locally integrable function. A notable borderline case is when the initial vorticity $\omega_0 = \alpha \delta_{x=0}$ is a multiple of a Dirac mass; the corresponding solution is the \emph{Lamb-Oseen vortex}
\begin{equation}
    \label{eq:lamboseen}
    \omega(x,t) = \frac{\alpha}{4 \pi t} e^{-\frac{|x|^2}{4t}} \, , 
\end{equation}
namely, $\alpha$ times the 2D heat kernel.\footnote{Recall that the non-linearity vanishes on radial solutions to~\eqref{eq:vorticityequation}.} The space $\mathcal{M}(\R^2)$ of finite Radon measures is a critical space in 2D. Remarkably,~\eqref{eq:lamboseen} is \emph{unique} in the class $C((0,T];L^1) \cap C_w([0,T];\mathcal{M})$, even among non-radial solutions, and globally attracts solutions with circulation $\alpha$~\cite{GallayWayne} (see also~\cite{GallagherGallay,GallayGallagherLions}). Until now, this was the only known large self-similar solution in 2D. Other vorticities may arise as principal value integrals
\begin{equation}
    \omega_0 = {\rm pv} \, \frac{1}{|x|^2} b_0 \left( \frac{x}{|x|} \right) \, ,
\end{equation}
where $b_0 : S^1 \to \R$ is a mean-zero function on the sphere, or as linear combinations with the Dirac delta. The conclusion is that the standard vorticity estimates are not generally applicable to self-similar solutions. \\

Our proof of Theorem~\ref{thm:main} is based on \emph{a priori estimates specific to self-similar solutions}, especially
\begin{equation}
    \label{eq:aprioriestimateinintro}
\| \langle y \rangle U \|_{L^\infty} \leq C(\alpha,\mathcal{M}) \, .    
\end{equation}

We begin with the basic energy estimate for~\eqref{eq:LerayEquations}:
\begin{equation} \label{eq:Dirichletenergyinintro}
	\int_{\R^2} |\nabla U|^2 \, dy \leq \frac{1}{4} \limsup_{R \to +\infty} R \int_{S_R} |U(y)|^2 \, ds \lesssim \| u_0|_{S_1} \|_{L^2(S_1)} \, 
\end{equation}
(see Lemma~\ref{b-en-global}). 
Then, based on the equation for the vorticity profile~$\Omega=\nabla^\perp \cdot U$, we establish
\begin{equation}
 \label{vet-1}
			\int_{\R^2} |\nabla\Omega|^2 dx=\int_{\R^2} |\Delta U|^2 dx \lesssim \M^2
		\end{equation}
(see Lemma~\ref{vort-est}). Proceeding further, we face the fundamental difficulty that in 2D the Dirichlet integral  does not control the  magnitude of the velocity itself (for example, $\|U\|_{L^2(B_1)}$), unlike in 3D where $\dot H^1(\R^3) \hookrightarrow L^6(\R^3)$. Moreover, functions in $\dot H^1(\R^2)$ can grow logarithmically as $|y| \to +\infty$ (consider, for example, $(\ln (2+|y|))^\frac13$), which is far from the expected $O(\frac{1}{|y|})$ decay in \eqref{eq:aprioriestimateinintro}. This low-frequency problem is typical for the steady Navier--Stokes equations in 2D \cite{GaldiBook} and has finally been overcome rather recently in various settings, see \cite{GKR23,korob-ren25} and references therein. However, the delicate arguments of \cite{GKR23,korob-ren25} are not  applicable to the steady-type system~(\ref{eq:LerayEquations}) because of the large additional term~$y \cdot \nabla_y U$, i.e.,  some new ideas are required, tailored to the  specific structure of \eqref{eq:LerayEquations}. 

Fortunately, in 2D we have a kind of ``magic wand" to control the pressure pointwisely, namely, the Poisson equation
\begin{equation}
    - \Delta P = \nabla U : (\nabla U)^T \, ,
\end{equation}
the same as for the steady Navier--Stokes system. By compensated compactness (div-curl lemma),
 the right-hand side  is estimated in the Hardy space $\mathcal{H}^1$, and, consequently, $\|\nabla P\|_{L^2}+\|P\|_{L^\infty}$ is bounded by $\M^2$ as well (see Lemma~\ref{lem:pressurebounds}).
 
 But before starting the machinery, we also need an estimate on $U$ in a neighbourhood of the origin. For the above reasons, this problem is not trivial  and cannot be resolved by standard integral imbedding inequalities; a new tool is required.  In the steady Navier--Stokes theory, the Bernoulli pressure $\phi = \frac{|u|^2}{2} + p$ satisfies $(- \Delta + u \cdot \nabla)\phi = - |\omega|^2$; in particular, it is a~subsolution. The analogous quantity for the Leray equations~\eqref{eq:LerayEquations} is
\begin{equation}
\label{eq:Bernoullipressureforselfsim}
    \Phi := \frac{1}{2} |U|^2 - \frac{y}{2} \cdot  U + P \, ,
\end{equation}
which satisfies
\begin{equation}
    (- \Delta + \frac{1}{2}U - \frac{y}{2} \cdot \nabla) \Phi = - |\Omega|^2.
\end{equation}
By the maximum principle, we have
\begin{equation}
    \Phi(y) \leq \limsup_{R \to +\infty}  \sup_{S_R} \Phi \leq C(\| u_0|_{S_1} \|_{L^\infty}) \, , \quad \forall y \in \R^2 \, .
\end{equation}
Note that $\Phi$ controls $|U|^2$ up to the (bounded) pressure and the unsigned term $\frac{y}{2}\cdot  U$. From this, one can argue that there exists a good circle $S_{R^*}$ on which $|U|$ satisfies $|U| \leq C(\| u_0|_{S^1} \|_{L^\infty})$ (see Lemma~\ref{lem:local}). Notably, this is the first time the ``self-similar Bernoulli pressure" $\Phi$ has been used in the literature on forward self-similar solutions. A related quantity was exploited by~\cite{NecasRuzickaSverak,TsaiBackwardSelfSim} to exclude non-trivial \emph{backward} self-similar solutions. We also mention \cite{JaySelfsimHighdim} which uses $\Phi$ to exclude self-similar steady Navier--Stokes solutions in high dimensions.

Now,  the large additional term $y \cdot \nabla_y U$ becomes very helpful:
it allows us to make the key observation that $\partial_r(rU)$ is almost in $L^2(\R^2)$, up to a logarithmic term from 
$(U\cdot\nabla)U$. Through some technical steps this leads to the weighted integral estimate
	\begin{align} \label{eq:key-w2-UP}
		\left\||y|^{-\gamma_2} (|U|+|P|)\right\|_{L^{2-\gamma_1}(\R^2)} 
		&\le C(\M)
	\end{align} 
with $0< \gamma_1 \ll \gamma_2 \ll 1$, 
 while remaining in the ``steady setting" of system~(\ref{eq:LerayEquations}) 
(see Proposition~\ref{prop:elliptic}). The underlying reason for the decay of $U$ is that \emph{the scaling operator $\frac{1}{2}(1+y\cdot \nabla_y)$ effectively dominates at spatial infinity}, even under the modest starting estimate \eqref{eq:Dirichletenergyinintro}.

In the final stage, we pass back to the time-dependent equations to further bootstrap the solution via parabolic regularity theory. This is done in an annular region $\mathcal{A} \times [0,t_0]$ away from $x=0$, since the behavior of $U$ as $|y| \to +\infty$ corresponds to the behavior of $u(x,t)$ as $t \to 0^+$ at fixed radius $|x|$. At this stage, not only is it necessary to prove the pointwise estimate~\eqref{eq:aprioriestimateinintro}, but also one must obtain some compactness by estimating the rate of convergence to the far-field initial data, $|U - u_0| \lesssim \langle y \rangle^{-1-\alpha}$, which corresponds to H{\"o}lder continuity for the time-dependent solution $u$ in the annular region $\mathcal{A} \times [0,t_0]$. 
For this, we apply the very efficient and elegant `local-in-space smoothing' approach developed in~\cite{jiasverakselfsim}; see Section~\ref{sec:bootstrap}.

\subsection{Comparison with existing literature}

\subsubsection{3D self-similar solutions} To begin, we review the strategies for constructing self-similar solutions in 3D.

Small self-similar solutions can be constructed via the perturbation theory~\cite{CannoneMeyerPlanchonAutosimilaires,CannonePlanchonSelfsimilarSolutionsCPDE,GigaMiyakawaMeasuresInitialVorticity,BarrazaSelfsimilarWeakLp,KochTataruWellposednessNavierStokes} in essentially the same way in dimensions $d \geq 2$. The spatial asymptotics of such solutions were investigated in~\cite{BrandoleseFineProperties}.

Large solutions were constructed by Jia and \v{S}ver{\'a}k in~\cite{jiasverakselfsim} from $-1$--homogeneous $C^\alpha(\R^3 \setminus \{ 0 \})$ initial velocity fields $u_0$. They obtain \emph{a priori} estimates in the following way: First, $u$ is controlled in $L^2_{\rm uloc}$ via uniformly local energy estimates~\cite{LemarieRieussetFirstL2uloc,LemarieRieussetRecentDevelopments,SereginKikuchi,KwonTsaiL2uloc}. Second, the solution is estimated in a H{\"o}lder class in an annular region $\mathcal{A} \times [0,t_0]$, with bounds depending only on norms of $u_0$, via a new local-in-space smoothing estimate. This second step controls the tails of the self-similar profile $U$, as mentioned above (see also Section~\ref{sec:bootstrap}). The smoothing estimate is of independent interest and has been refined in~\cite{BarkerPrangeLocalizedSmoothingConcentration,KangMiuraTsaiLocallyL3,AlbrittonBarkerPrangeLocalizedSmoothingHalfSpace}.

After~\cite{jiasverakselfsim}, many works gave alternative constructions accommodating a variety of extensions: discretely self-similar (DSS) solutions,\footnote{This means that there exists $\ell_0 > 1$ such that $u_{\ell_0} = u$, where $u_{\ell_0}$ is defined via the scaling symmetry~\eqref{eq:scaling}.} the half-space, low regularity, and combinations thereof. Notable works in this direction include~\cite{tsaidiscretely} (DSS with $\ell_0 - 1 \ll 1$),~\cite{korobkovtsai} ($\R^3_+$ via original Leray contradiction argument),~\cite{lemarie2016} ($L^\infty(\R^3 \setminus B_1)$),~\cite{bradshawtsaiII} ($L^{3,\infty}$),~\cite{bradshawtsairot} (rotational self-similarity),~\cite{bradshawtsaibesov,globalweakbesov} ($\dot B^{-1+3/p}_{p,\infty}$, $3 < p <+\infty$), and~\cite{ChaeWolfDiscretelySelfSimilar,BradshawTsaiL2Loc} ($L^2_{\rm loc}$). See also the general references~\cite{JiaSverakTsaiHandbook,BradshawTsaiSurveyRecentResults}.\footnote{For large self-similar solutions in other viscous fluid equations, see, for example,~\cite{AlbrittonBradshawSQG} (critical SQG),~\cite{LaiMHDSelfSim} (MHD), and~\cite{LaiMiaoZhengSelfsimFractional} (NS).} The main \emph{a priori} estimates on which the above works are predicated hold for solutions to the time-dependent problem with $u_0$ in critical function spaces but not necessarily self-similar. For example,~\cite{bradshawtsaiII} is based on a Calder{\'o}n-type splitting~\cite{CalderonWeakSolutionsLp} which results in solutions satisfying
\begin{equation}
    u = e^{t\Delta} u_0 + v \, ,
\end{equation}
where $v$ has finite energy
\begin{equation}
    \| v(\cdot,t) \|_{L^2}^2 + \int_0^t \| \nabla v(s) \|_{L^2}^2 \, ds \leq C(\| u_0 \|_{L^{3,\infty}}) t^{1/2} \, , \quad \forall t > 0 \, .
\end{equation}
Such solutions were investigated in~\cite{BarkerSereginSverakWeakL3} for general $L^{3,\infty}$ velocity fields.

The direct analogues of the above estimates all fail in 2D. It is therefore interesting to speculate \emph{whether existence itself may fail for non-self-similar $u_0$ belonging to critical spaces like $L^{2,\infty}$.} Unlike what happens in $\R^{3}$, it is plausible that there
is no solution for some non-self-similar initial data in this space, which could even be rotationally self-similar~\cite{bradshawtsairot}. The self-similar Bernoulli pressure $\Phi$, which plays a crucial role in our proof, is quite specialized.

\subsubsection{More non-uniqueness} 

In~\cite{albritton2021non}, Bru{\`e}, Colombo, and the first author obtained non-uniqueness for Leray-Hopf solutions for the 3D Navier--Stokes equations with external force in $L^1_t L^{3-}_x$ with compact support in space. This was done by constructing an unstable self-similar solution in the form of an expanding vortex ring, supported by the force, within the axisymmetric-without-swirl class. In~\cite{albritton2021non}, non-uniqueness arises not from a bifurcation but rather from instability in self-similar variables, akin to~\cite[Proof of Theorem 4.1]{jiasverakillposed}. According to the first author, a variation of the construction in~\cite{albritton2021non} should work in 2D with forcing in $L^1_t L^{2-}_x$.

Recently, a computer-assisted proof was announced by T.~Hou, Y.~Wang, and C.~Yang in~\cite{hou2025nonuniqueness}, who numerically establish the existence and instability of a self-similar solution without a force.

We also mention certain contributions of convex integration to the Navier--Stokes uniqueness question, with no attempt to be exhaustive. In~\cite{BuckmasterVicolNavierStokesAnnals}, Buckmaster and Vicol prove non-uniqueness in the finite kinetic energy class. In~\cite{CheskidovLuoL2crit2023}, Cheskidov and Luo prove 2D non-uniqueness in $L^\infty_t L^p_x$ for every $p < 2$; in~\cite{CheskidovLuoSharp2022}, they demonstrate non-uniqueness in $L^2_t L^p_x$ for every $p < +\infty$ and dimension $d \geq 2$.

Finally, Coiculescu and Palasek recently prove non-uniqueness arising from large $\BMO^{-1}$ initial data in 3D~\cite{coiculescu2025non}. They speculate about the 2D case in~\cite[Remark 1.7]{coiculescu2025non}.

\subsection{Notations} We denote $B_R(y_0) = \{y \in \R^2 : |y-y_0|<R\}$ and $S_R(y_0) = \{y \in \R^2 : |y-y_0| = R\}$. For simplicity, we write $B_R, S_R$ for $B_R(0), S_R(0)$. 

Throughout the proof, $C(\M)$ will be a general constant depending only on $\M$, whose exact value can change from line to line. We write $A \lesssim B$ for $A \le C B$ with some absolute constant $C$, and write $A \lesssim_\varepsilon B$ for $A \le C(\varepsilon) B$. As usual, we set $\langle x \rangle = \sqrt{1+|x|^2}$.

We often express the $-1$-homogeneous initial data in terms of its value on the unit circle $a_0 = u_0|_{S_1}$, namely,
\begin{equation}
    u_0(x) = \frac{1}{|x|} a_0 \left(\frac{x}{|x|}\right).
\end{equation}

\section{Functional set-up}

We will work with the Banach space
\begin{equation}
	X_\alpha := \{ V \in L^\infty(\R^2) : \langle y \rangle^{1+\alpha} V \in L^\infty \, , \  \div V = 0 \},
\end{equation}
equipped with the norm
\begin{equation}
	\| V \|_{X_\alpha} := \| \langle y \rangle^{1+\alpha} V \|_{L^\infty(\R^2)} .
\end{equation}
Defining the linear operator
\begin{equation} \label{eq:Leray}
	 \mathcal{L} U :=  \Delta U + \frac{1}{2} (U + y \cdot \nabla_y U),
\end{equation}
the momentum equation of \eqref{eq:LerayEquations} can be written as
\begin{equation}
	-\mathcal{L} U + \mathbb{P} (U \cdot \nabla U) = 0 \, ,
\end{equation}
with $\mathbb{P} = \text{Id} - \nabla \Delta^{-1} \div$ being the usual Leray projector.  We set $ a = e^\Delta u_0$ and encode the leading order asymptotic of $U$ (or equivalently, the initial datum of $u$) into the decomposition
\begin{equation}
	U = a + V.
\end{equation}
For later usage, we collect some basic properties of $a$ here:

\begin{lemma} \label{lem:a}
	There holds $a \in C^\infty(\R^2)$, $\div a = 0$ and $\mathcal{L} a = 0$. Moreover, for any $\alpha \in (0,1)$, we have
	\begin{equation} \label{eq:a-u0}
		  \|\langle y \rangle^{1+\alpha}(a-u_0)(y) \|_{L^\infty(\R^2 \setminus B_1)} + \|\langle y \rangle a(y)\|_{L^\infty(\R^2)}  \lesssim \M,
	\end{equation}
	and
	\begin{equation} \label{eq:a-2}
		 \| \langle y \rangle^{1+\alpha} \nabla a(y)\|_{L^\infty(\R^2)} \lesssim \M.
	\end{equation}
\end{lemma}
\begin{proof}
	By definition, $a$ is clearly smooth and divergence free.  The fact $\mathcal{L} a = 0$ follows from the operator identity $\mathcal{L} e^\Delta = e^\Delta ( \mathcal{L} - \Delta)$ and the scale-invariance of $u_0$. Moreover, it is easy to check that for any $x\in \R^2$ and $t>0$,
	\begin{equation} \label{eq:etDeltau0}
		(e^{t\Delta} u_0) (x) = \frac{1}{\sqrt{t}} a\left( \frac{x}{\sqrt{t}} \right).
	\end{equation} 
	Noting that $u_0 \in C^\alpha(\R^2 \setminus \{0\}) \cap L^1_{\text{loc}}(\R^2)$, we have
	\begin{equation} \label{eq:etDeltau0-reg}
		\|e^{t\Delta} u_0\|_{C^\alpha_{\text{par}}((B_2 \setminus B_{1/2}) \times [0,1])} \lesssim \M,
	\end{equation}
	where $C^\alpha_{\text{par}}$ denotes the H\"older norm with respect to the parabolic distance
	\begin{equation*}
		d_{\text{par}}((x_1,t_1),(x_2, t_2)) :=\sqrt{ |x_1-x_2|^2 + |t_1-t_2|}.
	\end{equation*}
	In particular, for $|x| = 1$ and $t \in (0,1]$ there holds
	\begin{equation}
		\left|(e^{t\Delta} u_0) (x) - u_0(x)\right| \lesssim  \M \, t^{\frac{\alpha}{2}},
	\end{equation}
	which implies the first estimate in \eqref{eq:a-u0} via the identity \eqref{eq:etDeltau0}. The second estimate in \eqref{eq:a-u0} follows from the first one and the obvious  pointwise bound of $a$ in $B_1$. 
	
	In addition to \eqref{eq:etDeltau0-reg}, we also have the smoothing estimate
		\begin{equation} \label{eq:etDeltau0-reg-2}
		\|\nabla  e^{t\Delta} u_0\|_{L^\infty(B_2 \setminus B_{1/2})} \lesssim \M \, t^{-\frac12+\frac{\alpha}{2}}, \quad \forall t \in (0,1],
	\end{equation}
	which, together with the obvious  pointwise bound of $\nabla a$ in $B_1$, implies  \eqref{eq:a-2}.
	
\end{proof}

Due to \eqref{eq:Leray} and Lemma \ref{lem:a}, $V$ satisfies the perturbed Leray equations
\begin{equation}
	\label{eq:equationforV}
	\left\{
\begin{aligned}
	&- \Delta V - \frac{1}{2} (V + y \cdot \nabla_y V) + \mathbb{P} \left[(a+V) \cdot \nabla (a+V) \right]= 0, \\
	&\div V = 0 \, .
\end{aligned}\right.
\end{equation}
Introducing the nonlinear operator
\begin{equation}
	\label{eq:Kdef}
	\mathcal{K}V := \mathcal{L}^{-1}  \mathbb{P} [(a+V) \cdot \nabla (a+V) ],
\end{equation}
the system~\eqref{eq:equationforV} can be rewritten as 
\begin{equation}
	\label{eq:LeraySchauderEquation}
	V - \lambda \mathcal{K} V = 0
\end{equation}
with $\lambda = 1$. We will apply the Leray-Schauder fixed point theorem  to solve \eqref{eq:LeraySchauderEquation} in the Banach space $X_\alpha$, for which the following two results provide the key ingredients.

\begin{proposition}
	\label{prop:K}
	For any $\alpha \in (0,1)$, the nonlinear operator $\mathcal{K} : X_\alpha \to X_\alpha$ is well-defined, continuous and compact.
\end{proposition}

\begin{proposition}
\label{prop:fullaprioriestimate}
For any solution $V \in X_\alpha$ to the fixed point equation \eqref{eq:LeraySchauderEquation} with $\lambda \in [0,1]$, we have the uniform {a priori} estimate
	\begin{equation} \label{eq:fullaprioriestimate}
		\| V \|_{X_\alpha} \le C(\alpha, \M).
	\end{equation}
\end{proposition}

Proposition \ref{prop:K} will be proved in Section \ref{sec:K}, and Proposition \ref{prop:fullaprioriestimate} will be proved in Sections \ref{sec:elliptic}--\ref{sec:bootstrap}. Together, they immediately  imply Theorem \ref{thm:main}   via the classical Leray-Schauder theory (see, e.g.,  \cite[Theorem 1.2.17]{KorobBook}).

\section{Properties of \texorpdfstring{$\mathcal{K}$}{K}} \label{sec:K}

\begin{proof}[Proof of Proposition~\ref{prop:K}]
    First of all, we rigorously define the operator $\mathcal{L}^{-1}$ by writing down its integral representation.  The linear equation
    \begin{equation}
        -\mathcal{L}W = -\Delta W - \frac{1}{2} (W + y \cdot \nabla_y W) = F
    \end{equation}
    with $W \in X_\alpha$ can be rewritten as the inhomogeneous heat equation
    \begin{equation} \label{eq:w-heat}
        \partial_t w - \Delta w = f
    \end{equation}
    in $\R^{2+1}$ with
    \begin{equation}
        w(x,t) := \frac{1}{\sqrt{t}} W\left(\frac{x}{\sqrt{t}}\right)
    \end{equation}
    and
    \begin{equation}
        f(x,t) := \frac{1}{t\sqrt{t}} F\left(\frac{x}{\sqrt{t}}\right).
    \end{equation}
    By the definition of $X_\alpha$, for $|x|>\sqrt{t}$ we have
    \begin{equation}
        |w(x,t)| \lesssim \frac{1}{\sqrt{t}} \left( \frac{|x|}{\sqrt{t}}\right)^{-1-\alpha} \|W\|_{X_\alpha} \le t^{\frac{\alpha}{2}}|x|^{-1-\alpha} \|W\|_{X_\alpha},
    \end{equation}
    and in particular, $\lim_{t \to 0+} w(x,t) = 0$ for any $x \in \R^2 \setminus \{0\}$.   
    Solving \eqref{eq:w-heat} with zero initial datum, we get
    \begin{equation}
        W(x) = w(x,1) = \int_0^1 \int_{\R^2} \frac{1}{4 \pi (1-t)} e^{-\frac{|x-x'|^2}{4(1-t)}} f(x',t) dx'  dt,
    \end{equation}
    which then leads to the integral representation formula
    \begin{align} \label{eq:L-int-rep}
        [(-\mathcal{L})^{-1} F](x) =  \int_0^1 \int_{\R^2} \frac{1}{4 \pi (1-t) t^\frac32} e^{-\frac{|x-x'|^2}{4(1-t)}} F\left(\frac{x'}{\sqrt{t}}\right)  dx'  dt.
    \end{align}
    
    Based on \eqref{eq:L-int-rep}, we now  prove the boundedness of $\mathcal{K}: X_\alpha \to X_\alpha$. Writing $U = a+V$, by \eqref{eq:a-u0} we know that
    \begin{equation} \label{eq:U-upper}
        |U(x)| \lesssim \langle x \rangle^{-1} (\M+\|V\|_{X_\alpha}).     
    \end{equation}
    By definition of the Leray projector, there holds
    \begin{equation}
        \mathbb{P}(U \cdot \nabla U) = U \cdot \nabla U + \nabla P
    \end{equation}
    with
    \begin{equation} \label{eq:P-def-1}
         P = - \Delta^{-1} \div \div (U\otimes U).
    \end{equation}
    Using the identities
    \begin{equation}
        (U \cdot \nabla U)\left(\frac{x'}{\sqrt{t}}\right) = \sqrt{t} \, U\left(\frac{x'}{\sqrt{t}}\right) \cdot\nabla_{x'} U\left(\frac{x'}{\sqrt{t}}\right),
    \end{equation}
    \begin{equation}
        (\nabla P)\left(\frac{x'}{\sqrt{t}}\right) = \sqrt{t} \, \nabla_{x'} P\left(\frac{x'}{\sqrt{t}}\right) = -\sqrt{t} \, (\nabla \Delta^{-1} \div \div)_{x'} (U\otimes U) \left(\frac{x'}{\sqrt{t}}\right),
    \end{equation}
    and integration by parts, we obtain, for $i=1,2$,
    \begin{align} \label{eq:KV-rep}
        (\mathcal{K}V)_i &= \mathcal{L}^{-1} [\mathbb{P} (U \cdot \nabla U)]_i \nonumber \\
        &= -\sum_{j,k=1}^2 \int_0^1 \int_{\R^2} H_{ijk} (x-x',t) \, (U_j U_k)\left(\frac{x'}{\sqrt{t}}\right) dx' dt,
    \end{align}
    where we denote
    \begin{align}
        H_{ijk}(x,t) :=  (\delta_{ik}\partial_j-\partial^3_{ijk} \Delta^{-1}) \left(  \frac{1}{4 \pi (1-t) t} e^{-\frac{|x|^2}{4(1-t)}} \right).
    \end{align}
    Since the kernel of $\partial^3_{ijk} \Delta^{-1}$ decays like $\langle x \rangle^{-3}$, there holds 
    \begin{equation}
        \left|H_{ijk}\left(x,\frac12\right)\right| \lesssim \langle x \rangle^{-3}.
    \end{equation}
    Noticing that
    \begin{equation}
        H_{ijk}(x,t) = \frac{1}{4\sqrt{2} (1-t)^\frac32 t} H_{ijk} \left(\frac{x}{\sqrt{2(1-t)}}, \frac12\right),
    \end{equation}
    we get the precise upper bound
    \begin{equation} \label{eq:H-upper}
        |H_{ijk}(x,t)| \lesssim \frac{1}{(1-t)^\frac32 t} \left\langle \frac{x}{\sqrt{1-t}} \right\rangle^{-3}, \quad \forall x\in \R^2, \, t\in (0,1).
    \end{equation}
    By \eqref{eq:U-upper}, \eqref{eq:H-upper}  and Lemma \ref{lem:integral} proved below, we have 
    \begin{align} \label{eq:KV-extra-decay}
        |(\mathcal{K}V)(x)| &\lesssim \int_0^1 \int_{\R^2} \frac{1}{(1-t)^\frac32 t} \left \langle \frac{x-x'}{\sqrt{1-t}} \right \rangle^{-3} \left \langle \frac{x'}{\sqrt{t}} \right\rangle^{-2}  dx'dt \cdot (\M+\|V\|_{X_\alpha})^2 \nonumber \\
        &\lesssim \langle x \rangle^{-2} (\M+\|V\|_{X_\alpha})^2.
    \end{align}
    In particular, there holds
    \begin{align} \label{eq:KV-Xalpha}
        \|\mathcal{K}V\|_{X_\alpha} \lesssim (\M+\|V\|_{X_\alpha})^2,
    \end{align}
    which shows that $\mathcal{K}:X_\alpha \to X_\alpha$ is well-defined and bounded. 
    
    Next, we prove that $\mathcal{K}:X_\alpha \to X_\alpha$ is continuous. By \eqref{eq:KV-rep}, we can write
    \begin{align} \label{eq:KV-diff-rep}
        (\mathcal{K}V_1 - \mathcal{K} V_2)_i &= \sum_{j,k=1}^2 \int_0^1 \int_{\R^2} H_{ijk} (x-x',t) \, \left[(a+V_1)_j (V_1-V_2)_k\right]\left(\frac{x'}{\sqrt{t}}\right) dx' dt \nonumber \\
        &\quad + \sum_{j,k=1}^2 \int_0^1 \int_{\R^2} H_{ijk} (x-x',t) \, \left[(V_1-V_2)_j (a+V_2)_k\right]\left(\frac{x'}{\sqrt{t}}\right) dx' dt.
    \end{align}
    Similar to \eqref{eq:KV-Xalpha}, there holds
    \begin{equation}
         \|\mathcal{K}V_1 - \mathcal{K}V_2\|_{X_\alpha} \lesssim (\M +\|V_1\|_{X_\alpha}+\|V_2\|_{X_\alpha})\|V_1-V_2\|_{X_\alpha},
    \end{equation}
    which implies the claimed continuity.

    Finally, we prove the compactness of $\mathcal{K}: X_\alpha \to X_\alpha$. By \eqref{eq:U-upper}, there holds
    \begin{equation} \label{eq:U-Ls-1}
        \|U\|_{ L^s(\R^2)} \lesssim_{s} \M + \|V\|_{X_\alpha}, \quad \forall s \in (2, +\infty],
    \end{equation}
    hence by \eqref{eq:P-def-1} and the estimates for Riesz transforms, we obtain
    \begin{equation} \label{eq:P-Ls-1}
        \|P\|_{ L^s(\R^2)} \lesssim_s (\M + \|V\|_{X_\alpha})^2, \quad \forall s \in (1, +\infty).
    \end{equation}
    By \eqref{eq:KV-extra-decay} and the local regularity theory for $\mathcal{L}$ as an elliptic operator, we have the local estimates
    \begin{equation} \label{eq:KV-reg}
        \|\mathcal{K}V\|_{W^{1,s}(B_R)} \lesssim C(s, R, \M, \|V\|_{X_\alpha}), \quad \forall R>0, \ s \in (1, +\infty).
    \end{equation}
    For $s > 2$, we recall the compact embedding $W^{1,s}(B_R) \hookrightarrow L^\infty(B_R)$ for any finite $R$. Now, suppose that $V^{k}$, $k = 1,2,3,\dots$ are a sequence of functions with the norms $\|V^{k}\|_{X_\alpha}$ uniformly bounded. By \eqref{eq:KV-Xalpha}, \eqref{eq:KV-reg}  and the standard diagonal argument, there exist a sequence of integers $k_j \to +\infty$ such that $\mathcal{K} V^{k_j}$ converges locally uniformly to some function $W^\infty \in X_\alpha$. In particular, for any $n \ge 1$, we can find $j(n) \ge 1$ such that 
    \begin{equation} \label{eq:inside-conv}
        \|(\mathcal{K}V^{k_{j(n)}} - W^\infty) \cdot \mathbf{1}_{B_n}\|_{X_\alpha} \le \frac{1}{n}.
    \end{equation} 
     On the other hand, by \eqref{eq:KV-extra-decay}, we know that
     \begin{equation} \label{eq:outside-conv}
         \|(\mathcal{K}V^{k_{j(n)}} - W^\infty) \cdot \mathbf{1}_{\R^2\setminus B_n}\|_{X_\alpha} \lesssim n^{-1+\alpha} \left(\M+\sup_k \|V^k\|_{X_\alpha}\right)^2.
     \end{equation} 
    Combining \eqref{eq:inside-conv} and \eqref{eq:outside-conv}, we deduce $\mathcal{K}V^{k_j(n)} \to W^\infty$ in $X_\alpha$. This proves the claimed compactness.

\end{proof}

\begin{lemma} \label{lem:integral} For any $x \in \R^2$, there holds
    \begin{equation} \label{eq:integral}
        \int_0^1 \int_{\R^2} \frac{1}{(1-t)^\frac32 t} \left\langle \frac{x-x'}{\sqrt{1-t}} \right\rangle^{-3} \left\langle \frac{x'}{\sqrt{t}} \right\rangle^{-2}  dx'dt \lesssim \left\langle x \right\rangle^{-2}.
    \end{equation}
\end{lemma}

\begin{proof}
    Denote $I$ for the integrand in \eqref{eq:integral} and consider the following cases.

    \smallskip
    \textbf{Case 1.} $0 < t \le \frac12$ and $|x|>10$. In this case, we have
    \begin{align} \label{eq:case1-I}
        I \lesssim t^{-1} \langle x-x' \rangle^{-3}  \left\langle \frac{x'}{\sqrt{t}} \right\rangle^{-2},
    \end{align}
    and thus,
    \begin{align}
        \int_{B_{|x|/2}(x)} I dx' &\lesssim |x|^{-2} \int_{B_{|x|/2}(x) } \langle x-x' \rangle^{-3} dx' \lesssim |x|^{-2},\\
        \int_{B_{2|x|}(0) \setminus B_{|x|/2}(x)} I dx' &\lesssim t^{-1} |x|^{-3} \int_{B_{2|x|}} \left\langle \frac{x'}{\sqrt{t}} \right\rangle^{-2} dx' \lesssim |x|^{-3} \ln \left( \frac{|x|}{\sqrt{t}} \right), \\
        \int_{\R^2 \setminus B_{2|x|}} I dx' &\lesssim \int_{\R^2 \setminus B_{2|x|}} |x'|^{-5}  \lesssim |x|^{-3}.
    \end{align}
    Adding the above three estimates, we get 
    \begin{equation}
        \int_{\R^2} I dx' \lesssim |x|^{-2} + |x|^{-3} |\ln t|. 
    \end{equation}

    \smallskip
    \textbf{Case 2.} $0 < t \le \frac12$ and $|x| \le 10$. In this case, we have \eqref{eq:case1-I} and thus,
    \begin{align}
        \int_{B_{20}} I dx' &\lesssim t^{-1} \int_{B_{20}} \left\langle \frac{x'}{\sqrt{t}} \right\rangle^{-2} dx' \lesssim |\ln t|, \\
        \int_{\R^2 \setminus B_{20}} I dx' &\lesssim  \int_{\R^2 \setminus B_{20}} |x'|^{-5} dx' \lesssim 1.
    \end{align}
    Adding the above two estimates, we get
    \begin{equation}
        \int_{\R^2} I dx' \lesssim |\ln t|. 
    \end{equation}
    
    \smallskip
    \textbf{Case 3.} $  \frac12< t < 1$ and $|x|>10$. In this case, we have
    \begin{align} \label{eq:case3-I}
        I \lesssim (1-t)^{-\frac32} \left\langle \frac{x-x'}{\sqrt{1-t}} \right\rangle^{-3}  \langle x' \rangle^{-2},
    \end{align}
    and thus, 
    \begin{align}
        \int_{B_{|x|/2}} I dx' &\lesssim |x|^{-3} \int_{B_{|x|/2} } \langle x' \rangle^{-2} dx' \lesssim |x|^{-3} \ln |x|,\\
        \int_{B_{2|x|} \setminus B_{|x|/2}} I dx' &\lesssim (1-t)^{-\frac32} |x|^{-2} \int_{B_{2|x|}} \left\langle \frac{x-x'}{\sqrt{1-t}} \right\rangle^{-3} dx' \lesssim (1-t)^{-\frac12} |x|^{-2}, \\
        \int_{\R^2 \setminus B_{2|x|}} I dx' &\lesssim \int_{\R^2 \setminus B_{2|x|}} |x'|^{-5}  \lesssim |x|^{-3}.
    \end{align}
    Adding the above three estimates, we get
    \begin{equation}
        \int_{\R^2} I dx' \lesssim (1-t)^{-\frac12}  |x|^{-2}. 
    \end{equation}

       \smallskip
    \textbf{Case 4.} $  \frac12< t < 1$ and $|x| \le 10$. In this case, we have \eqref{eq:case3-I} and
    \begin{align}
        \int_{B_{20}} I dx' &\lesssim (1-t)^{-\frac32} \int_{B_{20}} \left\langle \frac{x-x'}{\sqrt{1-t}} \right\rangle^{-3} dx' \lesssim (1-t)^{-\frac12},  \\
        \int_{\R^2 \setminus B_{20}} I dx' &\lesssim  \int_{\R^2 \setminus B_{20}} |x'|^{-5} dx' \lesssim 1. 
    \end{align}
    Adding the above two estimates, we get
    \begin{equation}
        \int_{\R^2} I dx' \lesssim (1-t)^{-\frac12}. 
    \end{equation}

    \smallskip
    Finally, combining all the above cases, we conclude that
    \begin{align}
        \int_0^1 \int_{\R^2} I dx' dt \lesssim \int_0^1 \left(  (1-t)^{-\frac12} \langle x \rangle^{-2}  + |\ln t| \langle x \rangle^{-3}  \right) dt \lesssim \langle x \rangle^{-2}.
    \end{align}
\end{proof}

\section{A priori estimates for the Leray equation} \label{sec:elliptic}

As in Proposition \ref{prop:fullaprioriestimate}, we consider a solution $V \in X_\alpha$ to the fixed-point equation \eqref{eq:LeraySchauderEquation} with $\lambda \in [0,1]$.  Writing $U = a + V$, there holds
\begin{equation} \label{eq:Ulambda}
	\left\{\begin{aligned}
	&- \Delta U - \frac{1}{2} (U + y \cdot \nabla U) + \lambda U \cdot \nabla U + \nabla P = 0 \\
	&\div U = 0,
	\end{aligned}\right.
\end{equation}
The pressure $P$ here is determined as
\begin{align} \label{eq:P-def}
	 P &= \lambda (-\Delta)^{-1} \div \div (U \otimes U) \nonumber \\
     & = - \frac{\lambda }{2\pi} \int_{\R^2}  \log |y-y'| \, \div  \div  (U\otimes U) (y') dy'.
\end{align}
For simplicity, we are suppressing the dependence of $(U,P)$ on $\lambda \in [0,1]$. By $V \in X_\alpha$ and \eqref{eq:a-u0}, as $y \to \infty$ we have
\begin{equation} \label{eq:U-asymp}
    U(y)  =  \frac{1}{|y|} a_0\left(\frac{y}{|y|}\right) + O\left(\frac{1}{|y|^{1+\alpha}}\right).
\end{equation}
Here the constant in $O(\cdot)$ depends on $\M$ and $\|V\|_{X_\alpha}$. According to the standard elliptic regularity theory, $U$ is locally smooth, hence all the calculations in this section hold in the classical sense. 

Our goal in this section is to derive weighted a priori estimates on $U$ and $P$ capturing their quantitative decay in space.  Although the finiteness of the weighted norms in Proposition \ref{prop:elliptic} is easily guaranteed by \eqref{eq:U-asymp}, the nontrivial point here is that they can be bounded in terms of the given datum $a_0$.  This is crucial for the application of the Leray-Schauder theorem.

For definiteness, we fix two constants
\begin{equation}
	\gamma_1 = \frac{1}{100}, \quad \gamma_2 = \frac{1}{10}.
\end{equation}

\begin{proposition} \label{prop:elliptic}
 There holds
	\begin{align} \label{eq:key-weighted-UP}
		\left\||y|^{-\gamma_2} (|U|+|P|)\right\|_{L^{2-\gamma_1}(\R^2)} 
		&\le C(\M).
	\end{align} 
\end{proposition}

To prove this result, we have to exploit the inherent structure of the Leray equation \eqref{eq:Ulambda} with coefficient $\lambda$  and acquire a number of preliminary estimates first.

\begin{lemma}
	There holds
	\begin{equation} \label{eq:nablaU-bound}
		\|\nabla U\|_{L^\infty(\R^2)} \lesssim C(\alpha, \M, \|V\|_{X_\alpha}).
	\end{equation}
\end{lemma}

\begin{proof}
	We use a standard bootstrapping argument based on Stokes estimates  (see, for instance, \cite[Theorem IV.4.4 and Remark IV.4.2]{GaldiBook}).	First, we can rewrite \eqref{eq:equationforV} as
	\begin{equation} \label{eq:eq-V-2}
		\left\{\begin{aligned}
		&- \Delta V  + \nabla P =   \div \left( \frac{1}{2} y\otimes V - \lambda U \otimes U \right) - \frac12 V,  \\ 
		&\div V = 0.
		\end{aligned}\right.
	\end{equation}
	By \eqref{eq:U-asymp} and definition of $X_\alpha$, it is clear that
	\begin{equation}
		\left|\frac{1}{2} y\otimes V - \lambda U \otimes U\right| (y) \lesssim  C(\M, \|V\|_{X_\alpha}) \, \langle y \rangle^{-\alpha}.
	\end{equation}
	Then, local Stokes estimate gives, for any $y \in \R^2$ and $s \in (1, +\infty)$,
	\begin{equation} \label{eq:U-W1s}
		\| V\|_{W^{1,s}(B_1(y))} \lesssim_s C(\M, \|V\|_{X_\alpha}) \langle y \rangle^{-\alpha}.
	\end{equation}
	Using \eqref{eq:U-W1s}, we further deduce that
	\begin{equation}
		\left\|\div \left( \frac{1}{2} y\otimes V - \lambda U \otimes U \right) - \frac12 V \right\|_{L^s(B_1(y))} \lesssim_s C(\M, \|V\|_{X_\alpha}) \langle y \rangle^{1-\alpha}.
	\end{equation}
	Applying local Stokes estimate to \eqref{eq:eq-V-2} again, we get
	\begin{equation} \label{eq:U-W2s}
		\| V\|_{W^{2,s}(B_1(y))} \lesssim_s C(\M, \|V\|_{X_\alpha}) \langle y \rangle^{1-\alpha}.
	\end{equation}
	Taking $s$ large (depending on $\alpha$), then interpolation between \eqref{eq:U-W1s} and \eqref{eq:U-W2s} gives 
	\begin{equation}
		\|\nabla V\|_{L^\infty(\R^2)} \lesssim C(\alpha, \M, \|V\|_{X_\alpha}).
	\end{equation}
	Together with \eqref{eq:a-2}, we get the claimed estimate \eqref{eq:nablaU-bound}.
	
\end{proof}

\begin{lemma}[Basic estimate for the Dirichlet integral]\label{b-en-global} There holds
    \begin{equation} \label{eq:energy-global}
        \int_{\R^2} |\nabla U|^2 dy = \frac{1}{4} \int_{S_1} |a_0|^2 d\theta  \leq \frac{\pi}{2} \M^2.
    \end{equation}
\end{lemma}

\begin{proof}
     Testing \eqref{eq:Ulambda} by $U$ on the ball $B_R$, and using the fact that
    \begin{equation}
        \int_{B_R} (U+y\cdot \nabla U) \cdot U dy =  \frac{R}{2} \int_{S_R} |U|^2 ds,
    \end{equation}
    we get the local energy identity
    \begin{align} \label{eq:energy-BR}
        \int_{B_R} |\nabla U|^2 dy &=  \int_{S_R} \left( U \partial_r U + \frac{R}{4}  |U|^2 - \frac{\lambda}{2}    |U|^2 U_r - P U_r  \right) ds %\nonumber \\ 
        %&= \frac{1}{2} \left(\partial_r \int_{S_r} |U|^2 ds \right) \bigg|_{r=R} +  \int_{S_R} \left( \frac{R}{4}  |U|^2 - \frac{\lambda}{2}    |U|^2 U_r - P U_r  \right) ds
    \end{align}
    Here, we are denoting $\partial_r  = \frac{y}{|y|} \cdot \nabla$ and $U_r = U \cdot \frac{y}{|y|}$.

    By \eqref{eq:U-asymp} and \eqref{eq:nablaU-bound}, we have
    \begin{align} \label{eq:zeroth-bdry-term}
    	\limsup_{R \to +\infty} \int_{S_R} U \partial_r U ds < C(\alpha, \M, \|V\|_{X_\alpha}). 
    \end{align}
    Due to \eqref{eq:U-asymp}, there holds
    \begin{equation} \label{eq:first-bdry-term}
        \lim_{R \to +\infty} \int_{S_R} \frac{R}{4} |U|^2 ds = \frac{1}{4} \int_{S_1} |a_0|^2 d\theta \le \frac{\pi}{2} \M^2,
    \end{equation}
    and 
    \begin{equation} \label{eq:second-bdry-term}
        \lim_{R \to +\infty} \int_{S_R} |U|^2 U_r ds = 0.
    \end{equation}
   Since $P \in L^2(\R^2)$ (see \eqref{eq:P-Ls-1}), by the mean value theorem there exist a sequence of radii $\{R_k\}_{k=1,2,3,\cdots}$ such that $R_k \to +\infty$ and
    \begin{equation}
         R_k{\int_{S_{R_k}} |P|^2 ds} \to 0
    \end{equation}
    as $k \to +\infty$. Consequently, there holds
    \begin{align} \label{eq:third-bdry-term}
        \left|\int_{S_{R_k}} P U_r ds\right| \lesssim \frac{1}{R_k} \int_{S_{R_k}} |P| ds 
        \lesssim \frac{1}{\sqrt{R_k}} \left(\int_{S_{R_k}} |P|^2 ds \right)^\frac12 \to 0
    \end{align}
    as $R_k \to +\infty$. Hence, by taking $R = R_k$ with $k \to +\infty$ in \eqref{eq:energy-BR} and using \eqref{eq:zeroth-bdry-term}, \eqref{eq:first-bdry-term}, \eqref{eq:second-bdry-term}, \eqref{eq:third-bdry-term}, we obtain 
    \begin{equation} \label{eq:Dirichlet-finite}
    	\int_{\R^2} |\nabla U|^2 dy \le C(\alpha, \M, \|V\|_{X_\alpha}).
    \end{equation}
    
    Next, we remove the dependence on $\|V\|_{X_\alpha}$ in \eqref{eq:Dirichlet-finite}. Since we now know that $|\nabla U| + |P| \in L^2(\R^2)$, similar to  \eqref{eq:third-bdry-term}, there exist  a sequence of radii $\{\widetilde{R}_k\}_{k=1,2,3,\cdots}$ going to infinity such that
    \begin{equation} \label{eq:two-bdry-terms}
    	 \left|\int_{S_{\widetilde{R}_k}} P U_r ds\right| + \left|\int_{S_{\widetilde{R}_k}} U \partial_r U ds\right| \to 0.
    \end{equation}
    Taking $R = \widetilde{R}_k$ with $k \to +\infty$ in \eqref{eq:energy-BR} and using \eqref{eq:first-bdry-term}, \eqref{eq:second-bdry-term}, \eqref{eq:two-bdry-terms}, we get the desired estimate \eqref{eq:energy-global}.
    
\end{proof}

\begin{lemma}[Basic bounds on the pressure]
    \label{lem:pressurebounds}
     The pressure $P$ satisfies
     \begin{equation} \label{eq:Pbound}
    \| \nabla^2 P \|_{L^1(\R^2)} + \| \nabla P \|_{L^2(\R^2)} + \| P \|_{L^\infty(\R^2)} \lesssim \M^2,
    \end{equation}
    and
    \begin{equation} \label{eq:P-asymp}
    	\lim_{y \to \infty} P(y) = 0.
    \end{equation}
\end{lemma}
    
\begin{proof}
    By \eqref{eq:energy-global} and the classical div-curl lemma \cite{CoifmanLionsMeyerSemmes}, we have
    \begin{equation}
         \div \div (U\otimes U) = \sum_{i,j =1}^2 \p_i U_j \partial_{j} U_{i} \in \mathcal{H}^1(\R^2),
    \end{equation}
    where $\mathcal{H}^1$ denotes the Hardy space, and there holds
    \begin{equation}
        \|\div   \div   (U\otimes U)\|_{\mathcal{H}^1(\R^2)} \lesssim \int_{\R^2} |\nabla U|^2 dy \lesssim \M^2.
    \end{equation}
    Hence, by the Calder\'on-Zygmund theorem for Hardy spaces \cite{SteinBook}, there holds
     \begin{equation}
     	\| \nabla^2 P \|_{L^1(\R^2)} + \| \nabla P \|_{L^2(\R^2)} \lesssim \M^2.
     \end{equation}
     Further, the remaining claims of the lemma follow from the last estimate and 2D Sobolev embedding \cite{MazyaBook}. We also refer to \cite[Theorem 5.13]{BFbook2002} for a self-contained proof of the relevant estimates.
\end{proof}

\begin{lemma}[Local upper bound near the origin] \label{lem:local}
	There holds
	\begin{equation} \label{eq:local-bound}
		\max_{B_{10}} |U(y)| \le C(\M).
	\end{equation}
\end{lemma}

\begin{proof}
	Denote
	\begin{equation}
		\Phi = \frac{\lambda}{2} |U|^2 - \frac{y}{2} \cdot U + P
	\end{equation}
	for the Bernoulli pressure associated with \eqref{eq:Ulambda}, and denote
	\begin{equation}
		\Omega = \partial_{y_2} U_1 - \partial_{y_1} U_2
	\end{equation}
	for the vorticity function.   From \eqref{eq:Ulambda}, one can derive
	\begin{equation}
		-\Delta \Phi + (\lambda U -  \frac{y}{2}) \cdot \nabla \Phi = - \lambda \Omega^2 \le 0,
	\end{equation}
	which implies the one-sided maximum principle for $\Phi$, namely, for any $R>0$,
	\begin{equation} \label{eq:Phi-max}
		\max_{B_R} \Phi \le \max_{S_R} \Phi.
	\end{equation}
	Note that 
	\begin{equation}
		\limsup_{|y| \to \infty} |\Phi| \lesssim \M
	\end{equation} 
	holds due to the asymptotic behavior of $U$ and $P$ at infinity,  see \eqref{eq:U-asymp} and \eqref{eq:P-asymp}. Taking $R \to \infty$ in \eqref{eq:Phi-max}, we receive the pointwise upper bound on $\Phi$ given by
	\begin{equation}
		\max_{\R^2} \Phi \lesssim \M.
	\end{equation}
	Further, due to $\frac{\lambda}{2}|U|^2 \ge 0$ and \eqref{eq:Pbound}, we obtain
	\begin{equation} \label{eq:yU-upper}
		\max_{\R^2} (-\frac{y}{2}) \cdot U \le \max_{\R^2} |P| + \max_{\R^2} \Phi \lesssim \M^2+\M.
	\end{equation}

	Using \eqref{eq:energy-global}, the mean value theorem for integrals and Sobolev embedding on the circle, we can find a \textit{good circle} $S_{r_*}$ with $r_* \in (1,2)$ such that
	\begin{equation} \label{eq:osc-U-good-circle}
		\max_{y \in S_{r_*}} |U(y) - \bar{U}(r_*)| \lesssim \left(\int_{S_{r_*}} |\nabla U|^2 ds\right)^\frac12 \lesssim \M,
	\end{equation}
	here $\bar{U}$ being the average of $U$ over circles:
	\begin{equation} \label{eq:Ubar-def}
		\bar{U}(r) = \frac{1}{2\pi r} \int_{S_r} U ds.
	\end{equation}
	Suppose there exists a point $y_* \in S_{r_*}$ such that
	\begin{equation} \label{eq:assumption-U-large}
		|U(y_*)| \ge M (\M^2 + \M),
	\end{equation}
	where $M$ is a sufficiently large absolute constant determined by the arguments to follow. By \eqref{eq:osc-U-good-circle} and choosing $M$ large, for any $y \in S_{r_*}$ we have
	\begin{equation}
		|U(y) - U(y_*)| \le \frac{1}{10} |U(y_*)|, \quad |U(y)| \le 2 |U(y_*)|,
	\end{equation}
	and consequently
	\begin{equation} \label{eq:key-lower-bound}
		\frac{(U(y)-U(y_{*})) \cdot U(y)}{|U(y_*)|} \ge -\frac{1}{5} |U(y_*)|.
	\end{equation}
	Let 
	\begin{equation}
		y_{**} = -r_* \frac{U(y_*)}{|U(y_*)|} \in S_{r_*},
	\end{equation}
	then combining \eqref{eq:key-lower-bound} and \eqref{eq:assumption-U-large}, we deduce
	\begin{align}
		-\frac{y_{**}}{2} \cdot U(y_{**}) %&= \frac{r_*}{2} \frac{U(y_*) \cdot U(y_{**})}{|U(y_*)|} \nonumber \\
		&= \frac{r_*}{2} \left(\frac{(U(y_{**})-U(y_{*})) \cdot U(y_{*})}{|U(y_*)|} + |U(y_*)| \right) \nonumber \\
		&\ge \frac{r_*}{2} \left(-\frac{1}{5} |U(y_*)| + |U(y_*)|\right) \nonumber \\
		&\ge \frac{2 M}{5} (\M^2 + \M). 
	\end{align}
	By choosing $M$ large, the last estimate contradicts with \eqref{eq:yU-upper}. Hence, we have shown that \eqref{eq:assumption-U-large} is impossible, in other words, the pointwise estimate	    
	\begin{equation} \label{eq:good-rstar}
		\max_{S_{r_*}} |U| \lesssim \M^2 + \M
	\end{equation}
	 must hold.

	By \eqref{eq:energy-global}, \eqref{eq:good-rstar} and Sobolev embedding, we have
	\begin{equation}
		\|U\|_{L^N(B_{20})} \lesssim_N |\bar{U}(r_*)| + 	\|\nabla U\|_{L^2(B_{20})} \lesssim \M^2 + \M 
	\end{equation}
	for any $N< +\infty$. Hence,   for $1<s<2$,
	\begin{equation}
		\|U\cdot \nabla U\|_{L^s(B_{20})} \le \|U\|_{L^{\frac{2s}{2-s}}(B_{20})} \|\nabla U\|_{L^2(B_{20})}  \lesssim_s \M^3 + \M.
	\end{equation}
	By equation \eqref{eq:Ulambda} and local Stokes estimates, we obtain
	\begin{align}
		\|U\|_{W^{2,s}(B_{10})} &\lesssim_s \|U\|_{W^{1,s}(B_{20})} + \|U\cdot \nabla U\|_{L^s(B_{20})} + \| y \cdot \nabla U\|_{L^s(B_{20})} \nonumber \\ &\lesssim_s \M^3 + \M.
	\end{align}
	The last estimate implies \eqref{eq:local-bound} via Sobolev embedding.
	
\end{proof}

\begin{lemma}\label{vort-est}
There holds
		        \begin{equation} \label{eq:Delta-U-sq}
			\int_{\R^2} |\Delta U|^2 dx \lesssim \M^2.
		\end{equation}
\end{lemma}

\begin{proof}
	Taking the curl of equation \eqref{eq:Ulambda}, we obtain the associated vorticity equation
	\begin{equation} \label{eq:vorticity}
		-\Delta \Omega -  \Omega - \frac{y}{2} \cdot \nabla \Omega + \lambda U \cdot \nabla \Omega = 0.
	\end{equation}
	Testing \eqref{eq:vorticity} with $\Omega$ in the ball $B_R$ gives
	\begin{equation} \label{eq:vorticity-energy-BR}
		\int_{B_R} |\nabla \Omega|^2 dy = \frac12 \int_{B_R} \Omega^2 dy + \int_{S_R} \left( \partial_r \Omega + \frac14 |y| \Omega - \frac{\lambda}{2} U_r \Omega  \right)  \Omega \,  ds.  
	\end{equation}
	We would like to prove the qualitative decay of the boundary integral in \eqref{eq:vorticity-energy-BR} along some sequence of $R \to +\infty$. First, we rewrite it as
	\begin{align} \label{eq:rewrite-bdry-int}
		&\quad \ \int_{S_R} \left( \partial_r \Omega + \frac14 |y| \Omega - \frac{\lambda}{2} U_r \Omega  \right)  \Omega  ds \nonumber \\
		&= \frac{1}{2R} \left[\frac{d}{dr} \left(r \int_{S_r} \Omega^2 ds \right)\right]\bigg|_{r=R} + \left( \frac{R}{4} -  \frac{1}{R} \right) \int_{S_R} \Omega^2 ds -\frac{\lambda}{2} \int U_r \Omega^2 ds.
	\end{align} 
	For simplicity, denote 
	\begin{equation}
		J(r) = r \int_{S_r} \Omega^2 ds.
	\end{equation}
	and $J'$ for its derivative in $r$. By \eqref{eq:Ulambda}, we can take $R_0>10$ (the quantitative size of $R_0$ is not important)  such that
	\begin{equation} \label{eq:outside-R0} 
		\sup_{\R^2 \setminus B_{R_0}} |U| \le 1.
	\end{equation}
	For $R > R_0$,  by \eqref{eq:rewrite-bdry-int} and \eqref{eq:outside-R0}, we have
	\begin{equation} \label{eq:upper-A}
		\int_{S_R} \left( \partial_r \Omega + \frac14 |y| \Omega - \frac{\lambda}{2} U_r \Omega  \right)  \Omega  ds \le \frac{1}{2R} J'(R) + \frac{1}{8} J(R).
	\end{equation}
	   By \eqref{eq:energy-global}, we have
	\begin{equation} \label{eq:omega-sq}
		\int_{\R^2} \Omega^2 dy \lesssim \M^2 < +\infty,
	\end{equation} 
	hence by the mean value theorem for integrals there exists a sequence of radii $\{\rho_k\}_{k=1,2,3,\dots}$ such that $\rho_k \in (2^{2k}, 2^{2k+1})$ and
	\begin{equation}
		J(\rho_k) \lesssim \int_{B_{2^{2k+1}} \setminus B_{2^{2k}}} \Omega^2 dy,
	\end{equation}
	which converges to $0$ as $k \to +\infty$. For each $k$, we define $\widetilde{\rho}_k \in [\rho_k, \rho_{k+1}]$ according to the following three cases:
	\begin{enumerate}
		\item If $J'(\rho_{k+1}) \le 0$, we simply let $\widetilde{\rho}_k =\rho_{k+1}$.
		\item If $J'(\rho_{k+1}) > 0$ and the set $\mathcal{C} = \{r \in [\rho_k, \rho_{k+1}] : J'(r)=0\}$ is non-empty, then we let $\widetilde{\rho}_k = \sup \, \mathcal{C}$.
		\item If $J'(r) > 0$ for all $r \in [\rho_k, \rho_{k+1}]$, then we choose $\widetilde{\rho}_{k}$ such that $J'(\widetilde{\rho}_{k}) = (\rho_{k+1} - \rho_k)^{-1}(J(\rho_{k+1})-J(\rho_k)),$
		whose existence is guaranteed by the mean value theorem. 
	\end{enumerate}
	By construction it can be verified that, as $k \to +\infty$,
	\begin{equation} \label{eq:A-vanishing}
		J(\widetilde{\rho}_k) \le J(\rho_{k+1}) \to 0, \quad  \max \{J'(\widetilde{\rho}_k), 0\} \to 0.
	\end{equation}
	Hence, based on \eqref{eq:upper-A}, \eqref{eq:omega-sq} and \eqref{eq:A-vanishing}, and by taking $R = \widetilde{\rho}_{k+1}$ in \eqref{eq:vorticity-energy-BR} with $k \to +\infty$, we get
	\begin{equation} \label{eq:omega-energy}
		\int_{\R^2} |\nabla \Omega|^2 dy \le \frac12 \int_{\R^2} \Omega^2 dy  \lesssim \M^2.
	\end{equation}
	Since $-\Delta U = \nabla^\perp \Omega$ with $\nabla^\perp := (-\partial_2, \partial_1)$, \eqref{eq:omega-energy} is equivalent to the desired estimate \eqref{eq:Delta-U-sq}.
	
\end{proof}

Now, we are ready to present:

\begin{proof}[Proof of Proposition \ref{prop:elliptic}]
   
    By Sobolev embedding and scaling,  for any $R \ge 2$ and $N \ge 1$ we know that
    \begin{align} \label{eq:scaled-Sobolev}
        R^{-\frac{2}{N}}\left( \int_{B_R} |U|^{N} dy \right)^\frac{1}{N} &\lesssim_N \left(\int_{B_R} |\nabla U|^2 dy\right)^\frac12 + \bar{U}(R), 
    \end{align}
    with $\bar{U}$ defined in \eqref{eq:Ubar-def}.   By H\"older's inequality and \eqref{eq:energy-global}, we know that, for $R > 2$,
    \begin{align} \label{eq:UR-Urstar}
        |\bar{U}(R) - \bar{U}(2)| &\le \frac{1}{2\pi}  \int_{2}^R \int_0^{2\pi} |\partial_r  U| d\theta dr \nonumber \\
        &\lesssim \left( \int_{2}^R r^{-1} dr \right)^\frac12 \left( \int_{\R^2} |\nabla U|^2 dy \right)^\frac12  \nonumber \\
        &\lesssim  (\ln R)^\frac12 \M.
    \end{align}
    In view of Lemma \ref{lem:local}, from \eqref{eq:UR-Urstar} we further deduce
    \begin{equation} \label{eq:UR-bound}
        |\bar{U}(R)| \lesssim C(\M) \, (\ln R)^\frac12, \quad \forall R>2.
    \end{equation}
    Then, combining \eqref{eq:scaled-Sobolev}, \eqref{eq:UR-bound} and \eqref{eq:energy-global} we obtain
    \begin{equation} \label{eq:BR-UN}
        \int_{B_R} (R^{-\frac{2}{N}} |U|)^{N} dy \lesssim_N (C(\M))^N (\ln R)^{\frac{N}{2}}. 
    \end{equation}
    Hence, for any  $\gamma>0$ there holds
    \begin{align} \label{eq:weighted-U-N}
        \int_{\R^2 \setminus B_{1}} \left(|y|^{-\frac{2}{N} - \gamma} |U|\right)^N dy &= \sum_{k=0}^{+\infty} \int_{B_{2^{k+1}} \setminus B_{2^k}} \left(|y|^{-\frac{2}{N} - \gamma} |U|\right)^N dy \nonumber \\
        &\lesssim_{N} (C(\M))^N \sum_{k=0}^{+\infty} 2^{-\gamma N k} (k+1)^{\frac{N}{2}}   \nonumber  \\
        &\lesssim_{\gamma, N} (C(\M))^N.
    \end{align}
    Applying H\"older's inequality and using \eqref{eq:weighted-U-N}, \eqref{eq:energy-global}, we get
    \begin{align} \label{eq:mu-N-bound}
        &\quad \ \int_{\R^2 \setminus B_{1}} |y|^{-\frac{4+2\gamma N}{2+N}}|U \cdot \nabla U|^{\frac{2N}{2+N}} dy \nonumber \\
        &\lesssim \left( \int_{\R^2 \setminus B_1} \left(|y|^{-\frac{2}{N} - \gamma} |U|\right)^N dy \right)^{\frac{2}{2+N}}  \left( \int_{\R^2 \setminus B_1} |\nabla U|^2 dy \right)^{\frac{N}{2+N}} \nonumber \\
        &\lesssim_{\gamma, N}  (C(\M))^{\frac{N}{2+N}} \lesssim  C(\M).
    \end{align}
    Making the change of parameters
    \begin{equation}
        \varepsilon_1 = \frac{4}{2+N}, \quad \varepsilon_2 = \frac{4+2\gamma N}{2+N},
    \end{equation}
    we deduce from \eqref{eq:mu-N-bound} that, for $\varepsilon_1 \in (0,1)$ and $\varepsilon_2>\varepsilon_1$, 
    \begin{equation} \label{eq:UnablaU-weighted}
        \int_{\R^2 \setminus B_{1}} |y|^{-\varepsilon_2}|U \cdot \nabla U|^{2-\varepsilon_1} dy \lesssim_{ \varepsilon_1, \varepsilon_2} C(\M).
    \end{equation}
    Also, notice that for such $\varepsilon_1$ and $\varepsilon_2$, there holds
    \begin{align} \label{eq:DeltaU-weighted}
        \int_{\R^2 \setminus B_{1}} |y|^{-\varepsilon_2} |\Delta U|^{2-\varepsilon_1} dy &\lesssim \left(\int_{\R^2 \setminus B_1} |y|^{-\frac{2\varepsilon_2}{\varepsilon_1} } dy\right)^{\frac{\varepsilon_1}{2}} \left(\int_{\R^2 \setminus B_1} |\Delta U|^2 dy \right)^{\frac{2-\varepsilon_1}{2}} \nonumber\\ 
        &\lesssim_{\varepsilon_1, \varepsilon_2}  C(\M)
    \end{align}
    by \eqref{eq:Delta-U-sq}. Similarly, by \eqref{eq:Pbound} there holds
     \begin{align} \label{eq:nablaP-weighted}
        \int_{\R^2 \setminus B_{1}} |y|^{-\varepsilon_2} |\nabla P|^{2-\varepsilon_1} dy \lesssim_{\varepsilon_1, \varepsilon_2}   C(\M).
    \end{align}
    Combining \eqref{eq:Ulambda} and \eqref{eq:UnablaU-weighted}--\eqref{eq:nablaP-weighted}, we arrive at
    \begin{align} \label{eq:key-weighted}
        &\quad \ \int_{\R^2 \setminus B_{1}} |y|^{-\varepsilon_2} |U + y \cdot \nabla U|^{2-\varepsilon_1} dy \nonumber \\
        &\lesssim \int_{\R^2 \setminus B_{1}} |y|^{-\varepsilon_2} \left(  |U \cdot \nabla U|^{2-\varepsilon_1} + |\Delta U|^{2-\varepsilon_1} + |\nabla P|^{2-\varepsilon_1}  \right) dy \nonumber \\
        &\lesssim_{\varepsilon_1, \varepsilon_2}  C(\M).
    \end{align}

    Now, it is convenient to work with polar coordinates $r = |y|$ and $\theta \in [0, 2\pi)$. Noticing that $U+y\cdot \nabla U = \partial_r (r U)$ and using \eqref{eq:local-bound}, \eqref{eq:key-weighted}, for $\rho>2$ we have
    \begin{align} 
        \rho |U(\rho, \theta)| &\le 2 |U(2, \theta)|  +  \int_{2}^{\rho} |U + y \cdot \nabla U| dr \nonumber \\
        &\le 2 |U(2, \theta)| + \left( \int_{2}^{\rho} r^{\frac{\varepsilon_2-1}{1-\varepsilon_1}} dr \right)^{\frac{1-\varepsilon_1}{2-\varepsilon_1}} \left(\int_{2}^{\rho}  r^{1-\varepsilon_2} |U + y \cdot \nabla U|^{2-\varepsilon_1} dr\right)^{\frac{1}{2-\varepsilon_1}}  \nonumber \\
        &\lesssim_{\varepsilon_1, \varepsilon_2} C(\M) + \rho^{\frac{\varepsilon_2-\varepsilon_1}{2-\varepsilon_1}} \left(\int_{2}^{\rho}  r^{1-\varepsilon_2} |U + y \cdot \nabla U|^{2-\varepsilon_1} dr\right)^{\frac{1}{2-\varepsilon_1}}. \nonumber
    \end{align}
    As a consequence, there holds
    \begin{equation} \label{eq:rhoU}
    	\sup_{r >2} r^{2-\varepsilon_2} |U(r, \theta)|^{2-\varepsilon_1} \lesssim_{\varepsilon_1, \varepsilon_2} C(\M) + \int_{2}^{+\infty}  r^{1-\varepsilon_2} |U + y \cdot \nabla U|^{2-\varepsilon_1} dr.
    \end{equation}
    Integrating \eqref{eq:rhoU} in $\theta$ and using \eqref{eq:key-weighted}, we obtain 
    \begin{align} 
    	\int_0^{2\pi}  \sup_{r>2} r^{2-\varepsilon_2}  |U(r, \theta)|^{2-\varepsilon_1} d\theta \lesssim_{\varepsilon_1, \varepsilon_2} C(\M),
    \end{align}
    which further implies
\begin{align} \label{eq:weighted-U-int}
	\int_{\R^2\setminus B_{2}} r^{-2\varepsilon_2} |U|^{2-\varepsilon_1} dy &\lesssim  \left(\int_{2}^{+\infty} r^{-1-\varepsilon_2} dr \right) \left(  \int_0^{2\pi}  \sup_{r>2} r^{2-\varepsilon_2}  |U(r, \theta)|^{2-\varepsilon_1} d\theta \right) \nonumber \\
	&\lesssim_{\varepsilon_1, \varepsilon_2} C(\M). 
\end{align} 
Combining \eqref{eq:weighted-U-int} with the local bound in Lemma \ref{lem:local} gives the desired estimate for $U$ in \eqref{eq:key-weighted-UP}.

Next, we prove the weighted estimate for pressure. A direct application of H\"older's inequality together with \eqref{eq:weighted-U-int} and \eqref{eq:weighted-U-N} gives
\begin{align} \label{eq:UU-weighted}
    \int_{\R^2 \setminus B_2} r^{-\delta_2} |U \otimes U|^{2-\delta_1} dy &\le \left( \int_{\R^2 \setminus B_2} r^{-2\varepsilon_2} |U|^{2-\varepsilon_1}  dy \right)^\frac{M-2}{M} \left(\int_{\R^2 \setminus B_2} r^{-2-{\nu M}} |U|^M dy \right)^{\frac{2}{M}} \nonumber \\
    &\lesssim_{\varepsilon_1, \varepsilon_2, \nu, M} C(\M),    
\end{align}
where we are taking $\delta_1, \delta_2>0$ as
\begin{equation} \label{eq:delta12}
	\delta_1 = \frac{\varepsilon_1}{2} + \frac{ 2-\varepsilon_1 }{M}, \quad \delta_2 = 2(\varepsilon_2+\nu)+ \frac{ 4-4\varepsilon_2}{M}.
\end{equation}
with arbitrary $M>10$ and $0<\nu<1$ (independent of $\varepsilon_1, \varepsilon_2$). Due to Lemma \ref{lem:local}, if $\delta_2$ is sufficiently small we can enhance  \eqref{eq:UU-weighted} to
\begin{eqnarray}
	 \int_{\R^2} r^{-\delta_2} |U \otimes U|^{2-\delta_1} dy \lesssim_{\varepsilon_1, \varepsilon_2, \nu, M} C(\M).
\end{eqnarray}
%\begin{equation}
%    \frac{M(2\delta_1 -2)}{M-2} = 2- \varepsilon_1, \quad \frac{M}{M-2} \left(-\delta_2+(2+\nu M)\frac{2}{M}\right)=-2\varepsilon_2.
%\end{equation}
Since $r^{-\delta_2}$ is an $A_{2-\delta_1}$-weight if $\delta_1, \delta_2$ are sufficiently small (see, for instance, \cite[Example 7.1.7]{GrafakosBook}), from the definition \eqref{eq:P-def} we deduce that
\begin{equation} \label{eq:weighted-int-P}
     \int_{\R^2} r^{-\delta_2} |P|^{2-\delta_1} dy \le C(\M).
\end{equation}
In view of \eqref{eq:delta12} and the requirement $\varepsilon_2>\varepsilon_1$, we point out that the allowed range of $(\delta_1, \delta_2)$ for \eqref{eq:weighted-int-P} to hold contains the set 
\begin{equation}
    \left\{(\delta_1, \delta_2) \in \R^2 : 0<\delta_1<\frac{1}{10}, \  4 \delta_1 <\delta_2<\frac12\right\}.
\end{equation}
This finishes the proof of the desired estimate for $P$ in \eqref{eq:key-weighted-UP}.

\end{proof}

\section{Bootstrapping via parabolic estimates} \label{sec:bootstrap}

In this section, we prove Proposition \ref{prop:fullaprioriestimate} using a bootstrapping procedure based on parabolic regularity theory and  the quantitative weighted estimates for $U, P$ obtained in Proposition \ref{prop:elliptic}.

\begin{proof}[Proof of Proposition \ref{prop:fullaprioriestimate}]
As before, we write $U = a + V$ which is a solution to \eqref{eq:Ulambda}. We separate the proof into a few successive steps.

    \smallskip
    \textbf{Step 1. Transforming to the non-stationary problem.}
    \smallskip

From $U, P$, we define
\begin{equation} \label{eq:up-def}
    u(x,t) = \frac{1}{\sqrt{t}} U\left(\frac{x}{\sqrt{t}}\right), \quad p(x,t) = \frac{1}{t} P\left(\frac{x}{\sqrt{t}}\right),
\end{equation}
then $(u,p)$ solves the non-stationary Navier--Stokes equations with coefficient $\lambda$, namely, 
\begin{equation} \label{eq:u-NS}
    \left\{\begin{aligned}
        &\partial_t u +  \lambda u \cdot \nabla u + \nabla p = \Delta u, \\
        &\nabla \cdot u = 0.
    \end{aligned}
    \right.
\end{equation}
By the local smoothness of $U$ and $P$, the time-dependent solution $(u, p)$ is  smooth for any positive time. By \eqref{eq:U-asymp}, the initial datum of $u$ is given by
\begin{equation} \label{eq:init-datum}
    u(x,0) = \frac{1}{|x|} a_0\left(\frac{x}{|x|}\right),
\end{equation}
which is achieved in the sense that, for any compact set $K \subset \R^2\setminus \{0\}$,
\begin{equation} \label{eq:init-conv}
    \sup_{x \in K} \sup_{0<t<1} t^{-\frac{\alpha}{2}}|u(x,t)-u(x,0)| < + \infty.
\end{equation}
Our goal is to obtain a quantitative version of \eqref{eq:init-conv} with the help of the available bounds on $\||y|^{-{\gamma_2}} P\|_{L^{2-{\gamma_1}}}$ and $\||y|^{-{\gamma_2}} U\|_{L^{2-{\gamma_1}}}$.

  By construction and Proposition \ref{prop:elliptic}, for any $t>0$ we have
\begin{equation}
    \left\|\left(\frac{|x|}{\sqrt{t}}\right)^{-{\gamma_2}} u(x,t)\right\|_{L^{2-{\gamma_1}}(\R^2)} = t^{\frac{{\gamma_1}}{4-2{\gamma_1}} }   \||y|^{-\gamma_2} U(y)\|_{L^{2-\gamma_1}(\R^2)} \le C(\M) \, t^{\frac{{\gamma_1}}{4-2{\gamma_1}} },
\end{equation}
and
\begin{align}
    \left\|\left(\frac{|x|}{\sqrt{t}}\right)^{-{\gamma_2}} p(x,t)\right\|_{L^{2-{\gamma_1}}(\R^2)} &=  t^{-\frac12 + \frac{{\gamma_1}}{4-2{\gamma_1}} }  \||y|^{-\gamma_2} P(y)\|_{L^{2-\gamma_1}(\R^2)} \nonumber \\
    &\le C(\M) \, t^{-\frac12 + \frac{{\gamma_1}}{4-2{\gamma_1}} }.
\end{align}
As a consequence, for the annulus domains $\A_R := B_R \setminus \overline{B_{\frac{1}{R}}} \subset \R^2$ with $R>1$, we obtain
\begin{equation} \label{eq:u-MU}
    \| u(x, t)\|_{L^{2-{\gamma_1}}(\A_R)} \les_{R} C(\M)  t^{ - \frac{{\gamma_2}}{2} + \frac{{\gamma_1}}{4-2{\gamma_1}}} ,
\end{equation}
and
\begin{equation} \label{eq:p-MP}
    \| p(x, t)\|_{L^{2-{\gamma_1}}(\A_R)} \les_{R} C(\M) t^{-\frac12 - \frac{{\gamma_2}}{2} + \frac{{\gamma_1}}{4-2{\gamma_1}}}.
\end{equation}

\smallskip
\textbf{Step 2. Propagating local-in-space-time energy estimates.}
\smallskip

Next, we perform a local energy estimate for $u$ inside the annulus $\A_5$ near the initial time. Let  $\varphi$ be a smooth cut-off function in $\R^2$ satisfying: (1) $\varphi(x) = 0$ for any $x \in \R^2 \setminus \A_5$; (2)  $\varphi(x)=1$ for any $x \in \A_4$; (3) $|\nabla \varphi(x)|  \lesssim |\varphi(x)|^{\frac{9}{10}}$ for any $x \in \R^2$. Testing the first equation of \eqref{eq:u-NS} with $u\varphi^2$, we obtain a local energy inequality given by
\begin{align} \label{eq:local-energy-est}
    &\quad \   \frac{d}{dt} \left(\int_{\R^2} |u|^2 \varphi^2 dx \right) +  2\int_{\R^2} |\nabla (u\varphi)|^2  dx \nonumber \\
    &=   \int_{\R^2} \left( 2 \lambda |u|^2 \varphi u \cdot \nabla \varphi  + 4 \varphi p u \cdot \nabla \varphi   + 2 |u|^2 |\nabla \varphi|^2  \right) dx.
\end{align}
By Sobolev embedding, for any $1\le N<+\infty$, we have
\begin{equation} \label{eq:sob}
    \|u\varphi\|_{L^N(\R^2)} \lesssim_N \left(\int_{\R^2} |\nabla (u \varphi)|^2 dx\right)^\frac12.
\end{equation}
Let us treat the right-hand side of \eqref{eq:local-energy-est} term by term. By the property (3) of $\varphi$, H\"older's inequality, \eqref{eq:u-MU} and \eqref{eq:sob}, we have
\begin{align} \label{eq:822-1}
    \left|\int_{\R^2} 2 \lambda |u|^2 \varphi u \cdot \nabla \varphi dx \right| &\lesssim \int_{\R^2}  |u|^3 \varphi^{\frac{19}{10}} dx \lesssim \int_{\R^2}  |u|^\frac32 |u \varphi|^\frac32  dx  \nonumber \\
    &\lesssim \left\|u\right\|^{\frac{3}{2}}_{L^{\frac{15}{8}}(\A_5)} \left\|u \varphi\right\|^{\frac{3}{2}}_{L^{\frac{15}{2}}(\R^2)}  \lesssim \left\|u\right\|^{\frac{3}{2}}_{L^{2-\gamma_1}(\A_5)} \left\|u \varphi\right\|^{\frac{3}{2}}_{L^{\frac{15}{2}}(\R^2)} \nonumber \\
    &\le C(\M) \, t^{-\frac{{3\gamma_2}}{4} + \frac{{3\gamma_1}}{8-4{\gamma_1}}} \left(\int_{\R^2} |\nabla (u \varphi)|^2 dx\right)^{\frac{3}{4}}.  
\end{align}
Hence, there holds
\begin{equation} \label{eq:nice-est-1}
	\left|\int_{\R^2} 2 \lambda |u|^2 \varphi u \cdot \nabla \varphi dx \right| - \frac{1}{3} \int_{\R^2} |\nabla (u \varphi)|^2 dx \le C(\M) \, t^{-3\gamma_2 + \frac{{3\gamma_1}}{2-{\gamma_1}}}.
\end{equation}
Applying H\"older's inequality, \eqref{eq:p-MP} and
 Ladyzhenskaya's inequality 
\begin{equation*}
    \|f\|_{L^4(\R^2)} \lesssim \|f\|_{L^2(\R^2)}^\frac12 \|\nabla f\|_{L^2(\R^2)}^\frac12, \quad \forall f \in C_0^\infty(\R^2),
\end{equation*}
we obtain
\begin{align}
    \left|\int 4\varphi p u \cdot \nabla \varphi dx \right|  &\lesssim \|p\|_{L^{\frac43}(\A_5)} \|u\varphi\|_{L^4(\R^2)} \lesssim \|p\|_{L^{2-\gamma_1}(\A_5)} \|u\varphi\|_{L^4(\R^2)} \nonumber \\
    &\le C(\M) \, t^{-\frac{1}{2}-\frac{\gamma_2}{2} + \frac{{\gamma_1}}{4-2{\gamma_1}}}  \left(\int_{\R^2} |u|^2 \varphi^2 dx\right)^{\frac14} \left(\int_{\R^2} |\nabla (u \varphi)|^2 dx\right)^{\frac14}. \nonumber 
\end{align}
Hence, there holds
\begin{equation} \label{eq:nice-est-2}
	\left|\int 4\varphi p u \cdot \nabla \varphi dx \right| - \frac{1}{3} \int_{\R^2} |\nabla (u \varphi)|^2 dx \le C(\M) \,  t^{-\frac{2}{3}-\frac{2\gamma_2}{3} + \frac{2\gamma_1}{6-3\gamma_1}} \left(\int_{\R^2} |u|^2\varphi^2 dx\right)^\frac13.
\end{equation}
For the last term in \eqref{eq:local-energy-est}, similar to \eqref{eq:822-1} we have
\begin{align}
     \int_{\R^2} 2|u|^2 |\nabla \varphi|^2 dx & \lesssim \int_{\R^2} |u|^2 |\varphi|^{\frac{9}{5}} dx \nonumber \\
    & \lesssim \|u\|_{L^1(\A_5)}^{\frac{1}{5}} \|u \varphi\|_{L^{\frac{9}{4}}(\R^2)}^{\frac{9}{5}} \lesssim \|u\|_{L^{2-\gamma_1}(\A_5)}^{\frac{1}{5}} \|u \varphi\|_{L^{\frac{9}{4}}(\R^2)}^{\frac{9}{5}}  \nonumber \\ 
    & \le C(\M) \, t^{ - \frac{{\gamma_2}}{10} + \frac{{\gamma_1}}{20-10{\gamma_1}}} \left(  \int_{\R^2} |\nabla (u \varphi)|^2 dx\right)^{\frac{9}{10}},
\end{align}
and consequently,
\begin{equation} \label{eq:nice-est-3}
	\int_{\R^2} 2|u|^2 |\nabla \varphi|^2 dx - \frac{1}{3} \int_{\R^2} |\nabla (u \varphi)|^2 dx \le C(\M) \,  t^{ -\gamma_2 + \frac{{\gamma_1}}{2-\gamma_1}}.
\end{equation}
Combining \eqref{eq:local-energy-est}, \eqref{eq:nice-est-1}, \eqref{eq:nice-est-2} and \eqref{eq:nice-est-3}, for $t \in (0,1)$ we get
\begin{align}
    &\quad \frac{d}{dt} \left(\int_{\R^2} |u|^2\varphi^2 dx\right) + \int_{\R^2} |\nabla (u \varphi)|^2 dx \nonumber \\
    &\le C(\M) \, t^{-\frac34} \left[\left( \int_{\R^2} |u|^2\varphi^2 dx\right)^\frac13 +1 \right].
\end{align}
Solving the above differential inequality together with the initial time bound
\begin{equation}
     \lim_{t \to 0+} \int_{\R^2} |u(x, t)|^2\varphi(x)^2 dx = \int_{\R^2} |u(x, 0)|^2\varphi(x)^2 dx \le C(\M),
\end{equation}
we arrive at the quantitative local energy bound
\begin{equation}
    \sup_{0\le t < 1} \int_{\R^2} |u|^2\varphi^2 dx + \int_0^1 \int_{\R^2} |\nabla (u \varphi)|^2 dx dt \le C(\M). 
\end{equation}

\smallskip
\textbf{Step 3. Regularity estimates near the initial time.}
\smallskip

Since the initial datum $u_0 = \frac{a_0\left(\frac{y}{|y|}\right)}{|y|} \in C^\alpha(\R^2 \setminus \{0\})$, similarly to \cite{Jiasver2013} we can decompose $u_0 = u_0^{(1)} + u_0^{(2)}$ with $\div u_0^{(1)} = 0$, $u_0^{(1)}|_{\A_{4}} = u_0$, $\supp u_0^{(1)} \subset \A_5$ and $\|u_0^{(1)}\|_{C^\alpha(\R^2)} \lesssim \M$.  Consider the local-in-time mild solution $v$ on a  time interval $[0,T]$ with $T = T(\alpha, \M) \in (0,1]$ to \eqref{eq:u-NS} with  initial datum given by $u_0^{(1)}$, namely,
\begin{equation} \label{eq:v-NS}
    \left\{\begin{aligned}
        &\partial_t v +  \lambda v \cdot \nabla v + \nabla q = \Delta v, \\
        &\nabla \cdot v = 0, \\
        &v(x, 0) =  u_0^{(1)}.
    \end{aligned}
    \right.
\end{equation}
The classical construction of $v$ (see, for instance, \cite{fujitakato, kato}) treats the nonlinear term $\lambda v \cdot \nabla v$ as a small perturbation to the linear Stokes system, hence $T$ can be chosen uniformly in $\lambda \in [0,1]$. Moreover, it is easy to show that $v$ enjoys the energy estimate
\begin{equation}
	 \sup_{0\le t \le T} \int_{\R^2} |v(x,t)|^2 dx + \int_0^{T} \int_{\R^2} |\nabla v|^2(x,t) dx dt  \le C(\M),
\end{equation}
and the regularity estimates
\begin{equation} \label{eq:vq-reg}
	\|v\|_{C^\alpha_{\text{par}}(\R^2 \times [0, T])}+ \|q\|_{L^\infty(\R^2\times [0, T])}  \le C(\alpha, \M).
\end{equation}

Let $u = v + w$, then $w$ satisfies the perturbed Navier--Stokes equations
\begin{equation} \label{eq:w}
    \left\{\begin{aligned}
            &\partial_t w -\Delta w + \lambda w \cdot \nabla w + \lambda v \cdot \nabla w + \lambda w \cdot \nabla v + \nabla \widetilde{p} = 0,    \\
            &\nabla \cdot w = 0
    \end{aligned}\right.
\end{equation}
in  $\R^2 \times [0, T]$, with $\widetilde{p} = p - q$ being the difference between the associated pressures for $u$ and $v$. Note also that the initial datum of $w$ enjoys the vanishing property
\begin{equation} \label{eq:w-init}
    w(x,0)|_{\A_4} = \big(u_0-u_0^{(1)}\big)|_{\A_4} \equiv 0.
\end{equation}
By the local energy estimate for $u$ obtained in Step 2 as well as the global-in-space energy estimate of $v$, we have
\begin{align}
    \sup_{0\le t \le T} \int_{\A_4} |w(x,t)|^2 dx + \int_0^{T} \int_{\A_4} |\nabla w|^2(x,t) dx dt  \le C(\M).
\end{align}
Standard interpolation gives
\begin{equation}
     \int_0^{T} \int_{\A_3} |w(x,s)|^4 dxds \le C(\M),
\end{equation}
and consequently, for $t\in [0,T]$,
\begin{equation}
    \int_0^{t} \int_{\A_3} |w(x,s)|^3 dxds \le  C(\M) \, t^{\frac14}.
\end{equation}
By \eqref{eq:p-MP} and \eqref{eq:vq-reg}, for $t \in [0, T]$ we have
\begin{align}
    \|\widetilde{p}(\cdot, t)\|_{L^{2-{\gamma_1}}(\A_3)} &\le \|p(\cdot, t)\|_{L^{2-{\gamma_1}}(\A_3)} + \|q(\cdot, t)\|_{L^{2-{\gamma_1}}(\A_3)} \nonumber \\
    &\le  C(\alpha, \M) \, t^{-\frac12 - \frac{{\gamma_2}}{2} },
\end{align}
%From \eqref{eq:w}, there holds
%\begin{equation}
%    \Delta q = - \div \div (w \otimes w + w \otimes v + v \otimes w)
%\end{equation}
and consequently,
\begin{equation}
    \int_0^{t} \int_{\A_3} |\widetilde{p}(x,s)|^{\frac32} dxds \le C(\alpha, \M) \, t^{\frac14 - \frac{3\gamma_2}{4} }.
\end{equation}
We extend the functions $w, \widetilde{p}, v$ by constant value $0$ to $\A_3 \times [-1+T, 0)$, then using \eqref{eq:w-init} one can verify that $w, \widetilde{p}$ is a suitable weak solution to the equations \eqref{eq:w} in the domain $\A_3 \times [-1+T, T]$.  Then, applying the $\varepsilon$-regularity result of \cite{jiasverakselfsim}, we deduce that $w \in C^\beta_{\text{par}}(\A_{2} \times (-\frac12, s])$ with some universal H\"older exponent $\beta>0$  for a sufficiently small $s_0 = s_0(\alpha, \M) \in (0, T]$, and there holds
\begin{equation}
    \|w\|_{C_{\text{par}}^\beta(\A_{2} \times (-\frac12, s_0])} \le C(\alpha, \M).
\end{equation}
We remark that the $\varepsilon$-regularity result of \cite{jiasverakselfsim} is for three-dimensional solutions, which include our two-dimensional solutions as a special case. Moreover, the additional parameter $\lambda \in [0,1]$ here is not harmful since the argument is perturbative.  By the regularity of $v$, we further deduce that
 \begin{equation} \label{eq:holder-1}
    \|u\|_{C_{\text{par}}^{\min\{\alpha,\beta\}}(\A_{2} \times [0, s_0])} \le C(\alpha, \M).
\end{equation}

\smallskip
\textbf{Step 4. Estimates in self-similar coordinates and the final bootstrap.}
\smallskip

By \eqref{eq:holder-1} and the definition of $u$, we have for $x \in \A_{2}$, $0 < t \le s_0$,
\begin{align}
   \biggl| \frac{1}{\sqrt{t}} U\left(\frac{x}{\sqrt{t}}\right) - \frac{1}{|x|} a_0\left(\frac{x}{|x|}\right)\biggr| \le C(\alpha, \M).
\end{align}
Letting $x = \frac{y}{|y|} \in S_1$ and $t = |y|^{-2}$, the above estimate implies that
\begin{align}
     \biggl| |y| U(y) -  a_0\left(\frac{y}{|y|}\right) \biggr| \le C(\alpha, \M) 
\end{align}
for $|y| \ge s_0^{-\frac12}$. Together with~\eqref{eq:local-bound}, we get
\begin{equation}\label{eq-cor-1}
    \bigl|U(y)\bigr| \lesssim C(\alpha, \M) \langle y \rangle^{-1} 
\end{equation}
for any $y \in \R^2$. By the Calder\'on-Zygmund theorem, we deduce
\begin{eqnarray}
    \|P\|_{L^{1+\varepsilon}(\R^2)} \lesssim_{\varepsilon} C( \alpha, \M)
\end{eqnarray}
for any $\varepsilon>0$.

Next, via a bootstrapping step we can improve the H\"older exponent in \eqref{eq:holder-1} to match the regularity of the initial data. (If $\beta \ge \alpha$,  this step is unnecessary.) By the definition of $p$, for any $\varepsilon>0$ and $0<t \le 1$ we have
\begin{equation} \label{eq:pL1+e}
    \|p(x,t)\|_{L^{1+\varepsilon}(\R^2)} = t^{-\frac{\varepsilon}{1+\varepsilon}}\|P\|_{L^{1+\varepsilon}(\R^2)} \le t^{-\varepsilon} C(\alpha, \M).
\end{equation}
Using  $\Delta p = -\lambda \div (u \cdot \nabla u)$, \eqref{eq:holder-1} and \eqref{eq:pL1+e}, there holds
\begin{equation} \label{pLinfty}
    \|p(x,t)\|_{L^{\infty}(\A_{\frac32})} \lesssim_{\varepsilon} t^{-\varepsilon} C(\alpha, \M)
\end{equation}
for $0< t \le 1$. Define 
\begin{equation}
    \widetilde{u} (x,t) = e^{t\Delta} u_0^{(1)} (x) - \int_0^t e^{(t-s)\Delta} (\lambda \phi u \cdot \nabla u + \phi \nabla p) (x,s) ds
\end{equation}
where $\phi$ is a smooth spatial cut-off function supported in $\A_{\frac32}$ with $\phi|_{\A_{\frac43}} = 1$, and let $\hat{u} = u-\widetilde{u}$. Then, by \eqref{eq:holder-1}, \eqref{pLinfty} and basic estimates for the heat kernel, we obtain
\begin{eqnarray}
    \|\widetilde{u}\|_{C_{\text{par}}^{\alpha}(\A_{\frac43} \times [0, 1])} \lesssim C(\alpha, \M).
\end{eqnarray}
Moreover, $\hat{u}$ solves the homogeneous heat equation $\partial_t \hat{u} = \Delta \hat{u}$ in $\A_{\frac43} \times (0,1]$ achieving zero initial datum at $\A_{\frac43} \times \{0\}$ and uniformly bounded at $\bigl(\partial \A_{\frac43}\bigr)\times[0,1]$ (see~(\ref{eq-cor-1})\,), hence it is locally smooth and, in particular,
\begin{equation}
    \|\hat{u}\|_{C_{\text{par}}^{\alpha}(\A_{\frac54} \times [0, 1])} \lesssim C(\alpha, \M).
\end{equation}
In conclusion, we have obtained the optimal local H\"older regularity for $u$, namely,
\begin{equation} \label{eq:u-Holder}
	\|u\|_{C_{\text{par}}^{\alpha}(\A_{\frac54} \times [0, 1])} \le C(\alpha, \M).
\end{equation}

Finally, going back to self-similar coordinates again, we get for $x \in \A_{\frac54}$, $0 < t \le 1$,
\begin{align}
    \biggl|\frac{1}{\sqrt{t}} U\left(\frac{x}{\sqrt{t}}\right) - \frac{1}{|x|} a_0\left(\frac{x}{|x|}\right)\biggr| \le C(\alpha, \M) t^{\frac{\alpha}{2}},
\end{align}
which then implies that, for $|y| \ge 1$,
\begin{align}
    \biggl|   |y| U(y) -  a_0\left(\frac{y}{|y|}\right)  \biggr| \le C(\alpha, \M) |y|^{-\alpha}.
\end{align}
Together with \eqref{eq:a-u0}, this proves the desired estimate \eqref{eq:fullaprioriestimate}.
\end{proof}

\section{Numerical observations}
\label{sec:numericalobservations}

In this section, we explain the numerical methods used to find Observations~\ref{numerics}.
The Leray equation~\eqref{eq:LerayEquations} admits the $\mathbb{Z}_{2}$-symmetry
$\mathcal{R}$ defined by the reflection with respect to the horizontal
line $x_{2}=0$.
We made the following more detailed numerical observations:

\begin{result}\label{numerics-detailed}
	For the $-1$--homogeneous initial velocity field,
	\begin{equation}
		u_{0}(y) = -\frac{\sigma y_{1}y}{|y|^{3}},
	\end{equation}
	we numerically observe the following:
	\begin{enumerate}
		\item In the range $\sigma\in[0,80]$, there exists a $\mathcal{R}$-symmetric
		solution $U_{\sigma}$ of \eqref{eq:LerayEquations} satisfying \eqref{eq:LerayEquationsBC}.
		\item There exists $\sigma_{0}\approx39.2$ such that the linearized operator
		around $U_{\sigma}$ has a one-dimensional kernel.
		\item At $\sigma=\sigma_{0}$, there is a supercritical pitchfork-type bifurcation
		corresponding to the breaking of the symmetry with respect to the
		horizontal axis. More precisely, for $\sigma\in[\sigma_{0},500]$,
		there exists two solutions $U_{\sigma}+V_{\sigma}$ and $U_{\sigma}+\mathcal{R}V_{\sigma}$
		of \eqref{eq:LerayEquations} satisfying \eqref{eq:LerayEquationsBC}, where $V_{\sigma}=0$
		at $\sigma=\sigma_{0}$ and $V_{\sigma}$ is not $\mathcal{R}$-symmetric
		(hence non trivial) for $\sigma>\sigma_{0}$.
	\end{enumerate}
\end{result}

Since the methods are pretty similar to what was done in 3D
\cite{guillodsverak}, the description is relatively brief.

\subsection{Discretization}

The computational domain is the ball of radius $R$ centered at the origin $B_{R}$, and its size is chosen to depend on $\sigma$ as $R_{\sigma}=100\sqrt{1+\sigma/80}$ in order to better track the solution.
For simplicity, instead of
considering \eqref{eq:LerayEquations} in $B_{R_{\sigma}}$ with $U(y)=u_{0}(y)$
on $\partial B_{R_{\sigma}}$, we rescale the amplitude and the domain
$B_{R_{\sigma}}$ to the unit ball $B_1$, hence we consider the
following problem to solve numerically:
\begin{equation}
	\label{eq:LerayEquationsBall}
	\left\{ \begin{aligned}-\Delta\tilde{U}-\frac{R_{\sigma}^{2}}{2}(\tilde{U}+\tilde{y}\cdot\nabla_{\tilde{y}}\tilde{U})+\sigma\tilde{U}\cdot\nabla\tilde{U}+\nabla\tilde{P} & =0 &  & \text{ in }B_1\\
		\div\tilde{U} & =0 &  & \text{ in }B_1\\
		\tilde{U}(\tilde{y}) & =\tilde{y}_{1}\tilde{y} &  & \text{ on }\partial B_1
	\end{aligned}
	\right.
\end{equation}
where
\begin{equation}
	\tilde{U}(\tilde{y})=\frac{1}{\sigma R_{\sigma}}U(R_{\sigma}\tilde{y}).
\end{equation}
That way, we work on the fixed domain $B_1$ on which only one mesh can be considered. The mesh is constructed by first meshing the disk uniformly,
then performing one refinement in $B_{\frac{1}{2}}$ and a second one in $B_{\frac{1}{4}}$, leading to a total of $100013$ cells and $50218$ vertices.
On this mesh, $\tilde{U}$ and $\tilde{P}$ are respectively discretized
using Lagrange P2 and P1 finite elements using the package FEniCS \cite{Logg-FENICS2012,Alnes-FENICS2015}.

\subsection{Continuation}

A numerical continuation is performed in $\sigma$: in $[0,10]$ we
choose a step of $0.02$ and in $[10,80]$ a step of $0.2$. At each
step the solution from the previous step is used as an initial datum
for a Newton's method. This Newton's method typically converges in
two or three steps and automatically remains in the class of $\mathcal{R}$-symmetric solutions.

We numerically find a solution $\tilde{U}_{\sigma}$ which is $\mathcal{R}$-symmetric
for $\sigma\in[0,80]$, represented on Figure \ref{fig:Usym}.

\subsection{Spectrum}

For each value of $\sigma$, the first 10 smallest eigenvalues of
the linearization of~\eqref{eq:LerayEquationsBall} around $\tilde{U}_{\sigma}$ are calculated
using the Krylov--Schur algorithm \cite{Hernandez-KrylovSchur2009} implemented
in SLEPc \cite{Hernandez-SLEPc2005}.

At $\sigma=0$, the eigenvalues found numerically are given up to
a very high precision by $1+\frac{n}{2}$, $n\in\N$ with
multiplicity $n+1$, and correspond exactly to the discrete spectrum
of the Stokes operator in self-similar coordinates.

The real part of the eigenvalues closer to the real axis are represented
on Figure~\ref{fig:spectrum}. In particular, a real eigenvalue crosses the
real axis near $\sigma_{0}\approx39.2$ whereas all the other eigenvalues
have a strictly positive real part on the range $\sigma\in[0,80]$.

Interestingly, the eigenvalue $\lambda=1$ is unchanged with respect
to $\sigma$; the explanation is that one can prove that any eigenvector
with non-zero mean has $\lambda=1$.

\subsection{Bifurcation}

Standard bifurcation theory indicates that when a real eigenvalue
is crossing the imaginary axis at $\sigma_{0}$, then another solution
should bifurcate from $\tilde{U}_{\sigma}$ at $\sigma=\sigma_{0}$.
This new branch of solution can be found numerically: for a value
of $\sigma$ slightly bigger than $\sigma_{0}$, Newton's iterations
are performed with the initial guess $\tilde{U}_{\sigma}+\alpha\phi$,
where $\phi$ is the eigenvector corresponding to the crossing eigenvalue,
and $\alpha\in\R$ is some real parameter to be adjusted.
Then a continuation described above can be performed.

This leads to the numerical construction of a non-symmetric solution
$\tilde{U}_{\sigma}+\tilde{V}_{\sigma}$ for $\sigma\in[\sigma_{0},80]$
represented on Figure \ref{fig:Uasym}.

\subsubsection*{Acknowledgments}
DA was supported by NSF Grant No. 2406947 and the Office of the Vice Chancellor for Research and Graduate Education at the University of Wisconsin–Madison with funding from the Wisconsin Alumni Research Foundation.
JG was supported by the French National Research Agency (Grants ANR-18-CE40-0027, project SingFlows and ANR-25-CE40-4532-09, project smasH) and by the Initiative d'Excellence (Idex) of Sorbonne University through the Emergence program. XR was supported by the National Key R\&D Program of China (No. 2023YFA1010700). DA and JG thank Vladim{\'i}r \v{S}ver{\'a}k for many fruitful discussions on non-uniqueness. DA also thanks Zachary Bradshaw and Tobias Barker for discussions on this problem.

\begin{remark}
    We recently learned that Changfeng Gui, Hao Liu, and Chunjing Xie are able to prove a result similar to Theorem~\ref{thm:main}~\cite{ChunjingXie}.
\end{remark}

\bibliographystyle{abbrvurl}
\bibliography{nonuniquenessbib}

\begin{figure}[p]
	\includegraphics{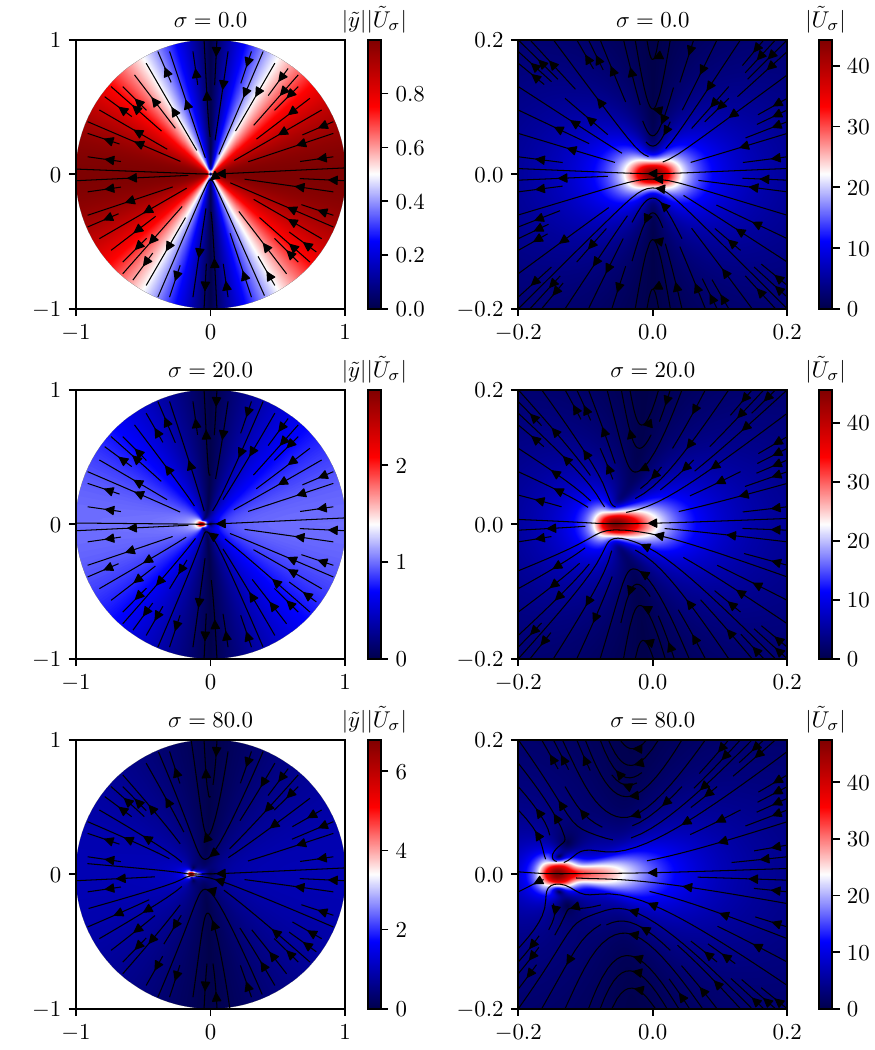}
	\caption{
		\label{fig:Usym}
		Representations of the solution $\tilde{U}_\sigma$ of \eqref{eq:LerayEquationsBall} for $\sigma\in\{0,20,80\}$: all solutions are $\mathcal{R}$-symmetric.
		The first column is on the full computational domain and represents in color the magnitude $|\tilde{y}||\tilde{U}(\tilde{y})|$
		and the streamlines of $\tilde{U}$. In particular, for $\sigma=0$, the solution is in addition symmetric with respect to the vertical axis because it solves the self-similar Stokes equations.
		The second column represent the solution near the origin with the color being the magnitude of $|\tilde{U}(\tilde{y})|$.
	}
\end{figure}

\begin{figure}[p]
	\includegraphics{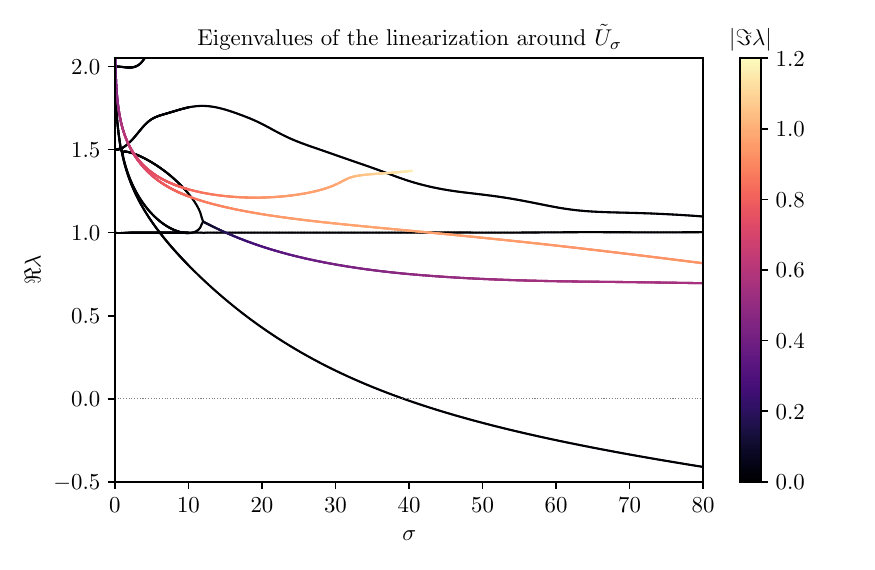}
	\includegraphics{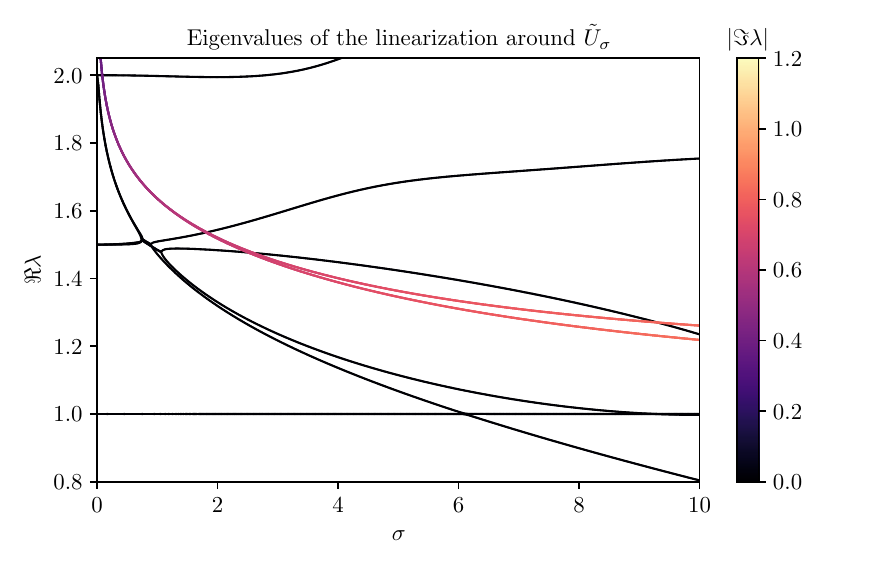}
	\caption{
		\label{fig:spectrum}
		Eigenvalues of the linearized operator around $\tilde{U}_\sigma$ found numerically.
		The color of the lines represents the absolute value of the imaginary part of the eigenvalues.
		The eigenvalue $\lambda=1$ is associated to the projection given by the mean of the eigenvectors.
		Near $\sigma_0\approx39.2$, one single real eigenvalue is crossing the imaginary axis.
	}
\end{figure}

\begin{figure}[p]
	\includegraphics{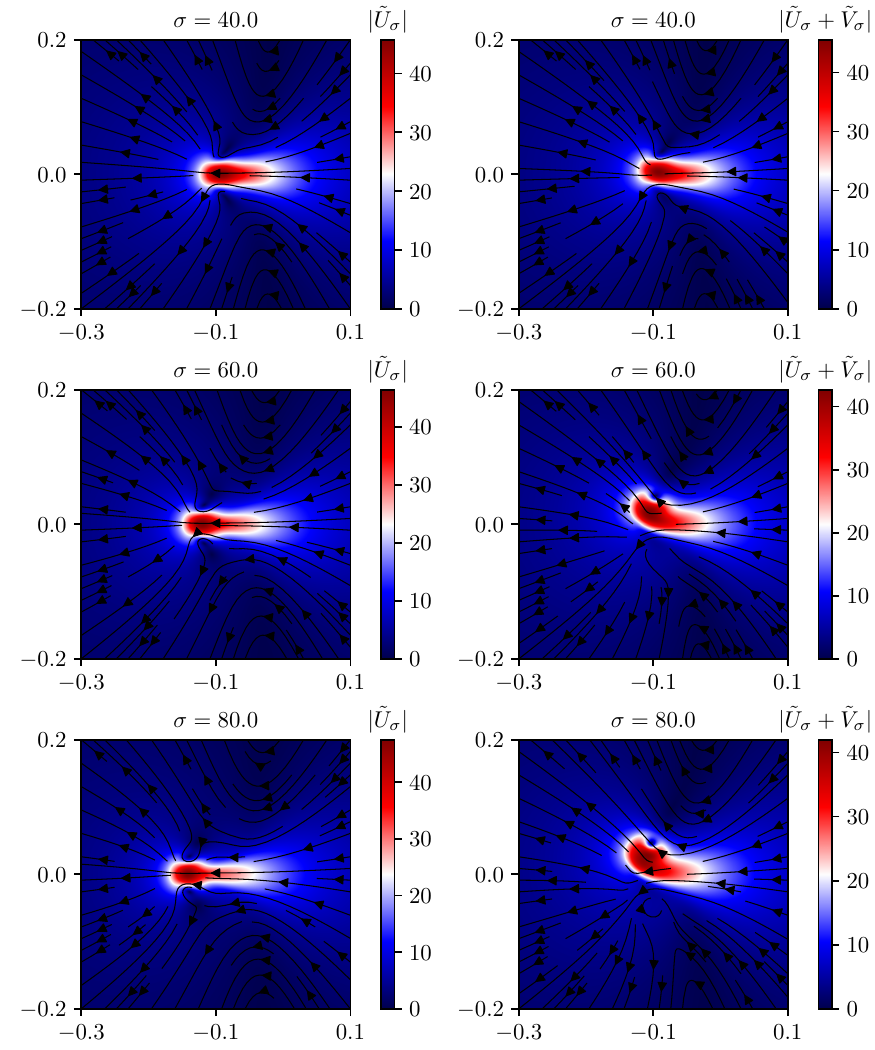}
	\caption{
		\label{fig:Uasym}
		Representation of the symmetric solution $\tilde{U}_\sigma$ on the left and of the non-symmetric solution $\tilde{U}_\sigma + \tilde{V}_\sigma$ on the right for $\sigma\in\{40,60,80\}$.
		The representation is done near the origin but slightly shifted to the left since the interesting part is in that direction.
		One visually sees that the two solutions are different.
	}
\end{figure}

\end{document}